\documentclass[preprint]{elsarticle}
\usepackage{amsmath,amsfonts,amssymb,amsthm}
\usepackage{bbm}
\usepackage{graphicx}

\usepackage{hyperref}
\usepackage{color}
\newcommand{\diff}{\mbox{\rm diff}}
\newcommand{\grad}{\nabla}

\newcommand{\laplace}{\Delta}

\renewcommand{\div}{\grad\cdot}
\newcommand{\N}{\mathbbm{N}}
\newcommand{\F}{{\mathcal F}}
\newcommand{\R}{\mathbbm{R}}
\renewcommand{\S}{\mathbbm{S}}

\newcommand{\Z}{\mathbbm{Z}}
\newcommand{\g}{\mbox{\sl g}}
\renewcommand{\Z}{\mathbbm{Z}}

\DeclareMathOperator{\spt}{supp}
\DeclareMathOperator{\Hess}{Hess}
\DeclareMathOperator{\hess}{hess}
\DeclareMathOperator{\argmin}{argmin}
\DeclareMathOperator{\supp}{supp}
\DeclareMathOperator{\gradient}{grad}
\def\loc{{\mathrm{loc}}}

\renewcommand{\L}{\ensuremath{\mathcal{L}}}
\newcommand{\Ha}{\ensuremath{\mathcal{H}}}
\newcommand{\D}{\ensuremath{\mathcal{D}}}

\renewcommand{\P}{\ensuremath{\mathcal{P}}}

\newcommand{\M}{\mathcal{M}}
\renewcommand{\H}{H^1_{\rho_*}}

\newcommand{\HaC}{\left.\Ha\right|_{C^{\infty}(\bar B_1)}}
\newcommand{\eps}{\varepsilon}

\newtheorem{remark}{Remark}
\newtheorem{theorem}{Theorem}
\newtheorem{proposition}{Proposition}
\newtheorem{lemma}{Lemma}

\newcommand{\tacka}{\, \cdot\,}

\newcommand{\pme}{porous medium equation}

\usepackage{graphicx}
\begin{document}

\title{Long-time asymptotics for the porous medium equation: The spectrum of the linearized operator\tnoteref{t1}}
\tnotetext[t1]{The author acknowledges support through NSERC grant 217006-08. }

\author{Christian Seis}
\ead{cseis@math.toronto.edu}

\address{University of Toronto,
Department of Mathematics,
40 St. George Street,
Toronto, Ontario M5S 2E4,
Canada,
Tel.: +1-416-946-3769,
Fax: +1-416-978-4107
}
\date{\notused}

\begin{abstract}
We compute the complete spectrum of the displacement Hessian operator, which is obtained from the confined porous medium equation by linearization around its stationary attractor, the Barenblatt profile. On a formal level, the operator is conjugate to the Hessian of the entropy via similarity transformation. We show that the displacement Hessian can be understood as a self-adjoint operator and find that its spectrum is purely discrete. The knowledge of the complete spectrum and the explicit information about the corresponding eigenfunctions give new insights on the convergence and higher order asymptotics of solutions to the porous medium equation towards its attractor. More precisely, the inspection of the eigenfunctions allows to identify symmetries in $\R^N$ with flows whose rates of convergence are faster than the uniform, translation-governed bound. The present work complements the analogous study of Denzler \& McCann for the fast-diffusion equation.
\end{abstract}


\maketitle

\section{Introduction}

In this paper, we study the long-time asymptotics of nonnegative solutions to the porous medium equation, i.e.,
\begin{equation}\label{eq1}
\partial_t \rho - \laplace(\rho^m)\;=\; 0\quad\mbox{in }\R^N,
\end{equation}
with exponent $m>1$. This equation is best known for modeling the flow of an isentropic gas through a porous medium; other applications include groundwater infiltration, population dynamics, and heat radiation in plasmas (cf.\ \cite[Ch.\ 2]{Vazquez07}).  

\medskip

Solutions to \eqref{eq1} feature different phenomena depending on the degree of the nonlinearity $\rho^m$. In the case $m=1$, equation \eqref{eq1} is the ordinary diffusion (or heat) equation. For $0<m<1$, the diffusion flux $m\rho^{m-1}$ diverges as $\rho$ vanishes and thus, for suitable initial data, the solution spreads over the whole $\R^N$ immediately. In this situation, equation \eqref{eq1} is often called the {\em fast diffusion} equation. In the porous-medium range $m>1$, the diffusion flux increases with the density and degenerates where $\rho=0$. Consequently, solutions will preserve a compact support and hence this type of propagation goes by the name {\em slow diffusion}. In the present paper, we restrict our attention exclusively to the latter case. For a study of the fast diffusion equation and more general evolutions of porous-medium type, we refer to {\sc V\'azquez} \cite{Vazquez06} and references therein.

\medskip

The Cauchy problem for the porous medium equation is solved in various settings. As the equation is degenerate parabolic, that is, it is parabolic only where the solution is positive, solutions are in general not classical. More precisely, if the initial datum is zero in some open subset of $\R^N$, then there is a slowly propagating free boundary that separates the region where the solution is positive from the region where it is zero. For suitable initial configurations, e.g.\ $0\le\rho_0\in L^1(\R^N)$, unique strong solutions are known to exist, and these solutions are bounded and continuous. Moreover, strong solutions preserve total mass, $\|\rho(t,\tacka)\|_{L^{1}(\R^N)}=\|\rho_0\|_{L^{1}(\R^N)}$ for all $t>0$. The Cauchy problem for the porous medium equation is reviewed by {\sc V\'azquez} in \cite{Vazquez07}.

\medskip

It is well-known \cite{ZeldovicKompaneec50,Barenblatt52,Prattle59} that the porous medium equation allows for self-similar solutions, so-called Barenblatt solutions, propagating on the scale $|x|=t^\alpha$ where
\[
\alpha\;=\; \frac1{N(m-1)+2},
\]
and given by
\begin{equation}\label{eq3}
\rho_*(t,x)\;=\; \frac1{t^{N\alpha}}\left(L-\frac{\alpha(m-1)}{2m} \frac{|x|^2}{t^{2\alpha}}\right)_+^{1/(m-1)}.
\end{equation}
Here, $(\cdot)_+=\max\{\tacka,0\}$ and  $L$ is an arbitrary constant that can be fixed, for instance, by normalizing the total mass or by choosing the radius of the support of $\rho_*$. It is interesting to note that although the Barenblatt solution is a strong solution of the porous medium equation, it is not a solution of the Cauchy problem as $\rho_*$ diverges to the delta measure $\delta_0$ (times a constant depending on $L$) at time zero --- a reason for which it is often called a ``source-type solution''.

\medskip

The Barenblatt solution describes the long-time behavior of any solution $\rho$ having same total mass as $\rho_*$. Indeed, for arbitrary non-negative initial data in $L^1$, solutions spread with a finite propagation speed over the whole space and, while diffusing, the shape of the solution becomes asymptotically close to the profile of the Barenblatt solution,
\[
\rho(t,\tacka)\approx\rho_*(t,\tacka)\quad\mbox{as }t\gg1.
\]
This long-time behavior was intensively studied over many years, starting with the work by {\sc Kamin} \cite{Kamin73,Kamin75}, who established uniform convergence in one space dimension. The generalization to several dimensions is due to {\sc Friedman} \& {\sc Kamin} \cite{FriedmanKamin80} and {\sc Kamin} \& {\sc V\'azquez} \cite{KaminVazquez88}. These authors use similarity (rescaling) transformations and identify the limit orbits for the porous medium equation via compactness arguments. A completely different approach relies on energy (or entropy) methods, a way prepared by {\sc Ralston} in \cite{Ralston84}, who established a first convergence result in $L^1$ in one dimension. This approach was generalized independently by {\sc Carrillo \& Toscani, Otto}, and {\sc Del Pino \& Dolbeault} in \cite{CarrilloToscani00,Otto01,DelPinoDolbeault02} to arbitrary dimensions. In these articles, instead of using geometric properties of the density function, the authors study the abstract energy landscape and compute the decay rates of the energy (or entropy) functional. In \cite{CarrilloToscani00} and \cite{Otto01}, these decay rates are converted into bounds on the asymptotics in $L^1$ and in the Wasserstein distance, respectively. We refer to {\sc V\'azquez}' excellent review article \cite{Vazquez03} for an almost complete discussion of the long-time asymptotics of solutions to the porous medium equation, including a discussion on the optimal rates of convergence. 

\medskip

The investigation of the optimal rate at which solutions converge to the Barenblatt solution is tied together with the choice of the initial data. The minimal requirement of integrability of the initial data --- which is from the physical point of view natural when thinking of $\rho_0$ as a concentration density --- is mathematically necessary as solutions to non-integrable initial data, i.e., $\rho_0\in L^1_{\loc}\setminus L^1$, show a different behavior in the long-time limit (cf.\ \cite[Cor.\ 1.2]{Vazquez03}). The authors in \cite{CarrilloToscani00,Otto01,DelPinoDolbeault02} derive algebraic rates of convergence for finite entropy solutions, cf.\ \eqref{eq41} below. The exponent in the algebraic rate is estimated by $\alpha$, which is optimal for spatial translations of the Barenblatt solution. The goal of the present article is to go one step further and to study the complete spectrum of a linearized version of the porous medium equation.  The knowledge of the spectrum and of all eigenfunctions not only yields information about the slowest converging modes (like a spectral gap estimate would do), but also allows to extract information about the characteristic geometric pathologies to all orders. In particular, we will study the role of translations, shears, and dilations for the convergence rates of solutions towards its attractor $\rho_*$. While spectral properties of the heat equation (i.e., \eqref{eq1} with $m=1$) are well-known in any space dimension, the investigation of the spectrum of the porous medium equation is new in the multi-dimension case. In the one-dimensional setting,  {\sc Angenent} \cite{Angenent88} computed the long-time asymptotics to all orders for solutions with compactly supported initial data. An earlier formal spectral analysis is due to {\sc Zel'dovic \& Barenblatt} \cite{ZeldovicBarenblatt58}. Progress in the analogous problem for the fast diffusion equation was recently obtained by {\sc Denzler, Koch \& McCann} \cite{DenzlerKochMcCann13}: the authors compute higher order asymptotics in weighted H\"older spaces, building up on earlier results for the linearized fast diffusion equation derived by  {\sc Denzler \& McCann} \cite{DenzlerMcCann05}.

\medskip

Our general approach to compute the spectrum of the linearized porous medium equation follows the one of Denzler \& McCann \cite{DenzlerMcCann05}, who studied the analogous problem for the fast diffusion equation. Both works are based on a method which is fairly common in the quantum mechanics literature, more precisely in the spectral analysis of Schr\"odinger operators, see e.g.\ \cite{NikiforovUvarov}. However, the present work and the one in \cite{DenzlerMcCann05} differ in some aspects. A main difference relies on the occurrence of a free boundary in the porous medium regime. In particular, the Barenblatt solution has compact support in $\R^N$, opposed to the situation for $m\le1$. This fact will play a crucial role in the definition of the linearized operator as (``asymptotic'') boundary conditions on the boundary of the support of the Barenblatt solution have to be taken into consideration.  Somehow surprisingly (and unexpected, see \cite[p.\ 303]{DenzlerMcCann05}), compared to the analogous study for the fast diffusion regime, we are able to simplify the analysis of our linearized operator {\em thanks to} the compact support of $\rho_*$ in many aspects: Parts of our arguments are built on elliptic (and spectral) theory of differential operators on bounded domains. The suitable elliptic theory is derived in the appendix of this paper.

\medskip

We caution the reader that our linearization of the porous medium equation performed in subsection \ref{S1.3} is only formally justified. In particular, the rigorous results obtained for the linearized operator have to be considered as conjectures for the asymptotics of solutions to the nonlinear equation \eqref{eq1}.


\subsection{Rescaling}

In order to investigate the asymptotic behavior of solutions to the porous medium equation, it is convenient to rescale variables in such a way that the Barenblatt solution becomes a stationary solution, that means $|x|\sim t^{\alpha}$ in view of \eqref{eq3}. One convenient choice of new variables is
\[
x = \beta t^{\alpha} \hat x,\quad t= \exp(\alpha^{-1}\hat t\,),\quad\mbox{and}\quad \rho=(\alpha \beta^2)^{1/(m-1)}  t^{- \alpha N}\hat \rho,
\]
with
\[
\beta=\left(\frac{2m L}{\alpha(m-1)}\right)^{1/2},
\]
to the effect of
\[
\rho(t,x)= \frac{(\alpha \beta^2)^{1/(m-1)}}{t^{ N\alpha}}\hat \rho\left(\alpha \ln t, \frac{x}{\beta t^{\alpha}}\right).
\]
We remark that the logarithmic rescaling of the time variable, $\hat t=\alpha \ln  t$, guarantees the parabolic structure of the equation. In fact, with its new variables, equation \eqref{eq1} turns into the nonlinear Fokker--Planck equation
\begin{equation}\label{eq2}
\partial_{\hat t}\hat \rho - \hat\grad\cdot(\hat x \hat\rho)- \hat\laplace(\hat\rho^m) \;=\;0,
\end{equation}
sometimes also called the {\em confined} \pme. Moreover, with that particular choice of $\beta$, the rescaled, time-independent Barenblatt solution $\hat\rho_*$ concentrates on a ball of radius one around the origin,
\begin{equation}\label{eq51}
\frac{2m}{m-1}\hat\rho_*(\hat x)^{m-1}\;=\;  \left(1-|\hat x|^2\right)_+.
\end{equation}
This normalization will simplify the notation in the subsequent sections. In the sequel, we will often refer to $\hat\rho_*$ as the {\em Barenblatt profile}.

\medskip

The advantage of working with \eqref{eq2} instead of \eqref{eq1} relies on the fact that the above rescaling comes along with a change of perspective: the Barenblatt solution now becomes a fixed point of the equation. In fact, as we will see below, for fixed total mass $\hat M$, it is both the unique stationary solution of \eqref{eq2} and the ground state of the physical entropy
\begin{equation}\label{eq41}
\hat E(\hat \rho)\;=\; \frac1{m-1}\int \hat\rho(\hat x)^m\, d\hat x + \frac{1}2 \int |\hat x|^2 \hat\rho(\hat x)\, d\hat x.
\end{equation}
In this regard, the porous medium equation \eqref{eq1} shows a substantially different behavior than the confined \pme\ \eqref{eq2}: There is no stationary solution of \eqref{eq1} under the mass constraint $\|\rho\|_{L^1}=M$ for any $M>0$. In other words, solutions to the porous medium equation do not relax in finite time and the ground state can be attained only asymptotically in the long-time limit.

\medskip

Somehow related to this feature is the fact that, opposed to the original equation \eqref{eq1}, the confined \pme\ \eqref{eq2} is no longer invariant under spatial translations. More precisely, in the new variables the origin is the unique attraction point in $\R^N$ around which the mass asymptotically concentrates, driven by the convection term $-\hat\grad\cdot(\hat x \hat\rho)$. Due to this loss of translation invariance it is therefore not surprising that those initial data $\hat\rho_0$ that differ from the ground state $\hat\rho_*$ only by a shift, $\hat\rho_0(\hat x) = \hat\rho_*(\hat x- \hat x_0)$ for some $\hat x_0\in\R^N$, will play a distinguished role in the discussion of our result on refined asymptotics of \eqref{eq2}: Spatial translations correspond to the smallest eigenvalue of the linearized operator and generate thus those solutions that saturate the optimal rate of convergence.

\medskip

The above rescaling, however, also has a small drawback: It is singular at time zero. To compare ``initial data'', one would better consider the moment $t=1$. More precisely, up to spatial dilations both in the $x$ and $\rho$ variables, the initial datum $\hat\rho(\hat t= 0,\tacka)$ corresponds to the original solution at time one $\rho(t=1,\tacka)$. Because of our interest in the long-time dynamics, this skew correspondence, however, has no influence neither on the subsequent analysis nor on the interpretation of our result.

\medskip

It remains to discuss the stationarity of $\hat\rho_*$ for the nonlinear Fokker--Planck equation \eqref{eq2} and its minimality for the entropy \eqref{eq41}. Calculating the energy dissipation along trajectories of the flow \eqref{eq2}, 
\[
\frac{d}{d\hat t} \hat E(\hat\rho( \hat t))\;=\; -\int \hat \rho\left|\hat\grad\left(\frac12|\hat x|^2 + \frac{m}{m-1}\hat\rho^{m-1}\right)\right|^2\, d\hat x,
\]
we see that the dissipation rate is zero if and only if $\frac12|\hat x|^2 + \frac{m}{m-1}\hat\rho^{m-1}$ is constant, that is, if and only if $\hat \rho$ is of the form \eqref{eq51}, provided that $\hat \rho$ and $\hat \rho_*$ have the same total mass. This implies that $\hat\rho_*$ is the unique stationary solution of \eqref{eq2}. To deduce the minimality of $\hat\rho_*$ for the entropy $\hat E$, we additionally observe that the entropy is strictly convex on the convex configuration space of nonnegative densities with fixed total mass $\hat M$ and admits thus a unique minimizer.

\subsection{Otto's gradient flow interpretation and the Wasserstein distance}\label{SO}

In his seminal paper \cite{Otto01}, {\sc Otto} introduced a new and physically meaningful gradient flow interpretation of the porous medium equation \eqref{eq1}. This interpretation is based on the Lyapunov approach of {\sc Newman} \& {\sc Ralston} in \cite{Newman84,Ralston84}. In the abstract setting, a gradient flow is the dynamical system on a Riemannian manifold $(\M, \g)$ that describes the evolution as the steepest descent in an energy landscape, i.e.,
\begin{equation}\label{eq31}
\partial_t\rho + \gradient E(\rho)\;=\;0.
\end{equation}
Here, the gradient $\gradient E$ of the energy $E$ is defined via Riesz' representation theorem
\[
\diff E(\rho).\delta \rho\;=\; \g_{\rho}(\gradient E(\rho),\delta \rho)\quad \mbox{for all }\rho\in T_{\rho}\M,
\]
where $\diff E$ denotes the differential of $E$ and $T_{\rho}\M$ is the tangent plane at the point $\rho\in\M$. A short formal computation shows that \eqref{eq31} can be equivalently stated as
\begin{equation}\label{eq31a}
\partial_t\rho\;=\;\argmin \left\{ \frac12\g(\delta\rho, \delta\rho) + \diff E(\rho).\delta\rho\: :\: \delta\rho\in T_{\rho}\M\right\}.
\end{equation}
In the following, we identify the gradient flow ingredients $\M$, $\g$, and $E$ that constitute the porous medium equation using Otto's approach. (In fact, different gradient flow interpretations were proposed for the porous medium equation. Otto's gradient flow is natural in view of the thermodynamical background of the equation. Moreover, it also pertains to the confined evolution, whereas the traditional approaches do not.)

\medskip

As solutions to \eqref{eq1} preserve non-negativity (in fact, solutions satisfy a comparison principle) and the total mass $M$, the manifold in the gradient flow interpretation of the porous medium equation has to be chosen as
\[
\M\;=\; \left\{ \rho\: :\: \rho\ge 0\mbox{ and }\int\rho\, dx=M\right\}.
\]
Tangent fields $\delta\rho$ at $\rho\in\M$ must be mass-preserving and non-negative where $\rho=0$. For tangent fields satisfying $\spt(\delta \rho)\subset\spt(\rho)$, the metric tensor is defined by
\[
\g_{\rho}(\delta\rho,\delta\rho)\;=\; \int \rho |\grad \psi|^2\, dx,
\]
where $\psi$ and $\delta\rho$ are related via the elliptic boundary value problem
\begin{equation}\label{eq16}
\begin{array}{rcll}
-\div\left(\rho\grad \psi\right) &=& \delta \rho&\quad\mbox{in }\spt(\rho),\\
\rho \grad \psi\cdot\nu &=&0&\quad\mbox{on }\partial\spt(\rho).
\end{array}
\end{equation}
Here, $\nu$ denotes the outer normal vector on $\partial\supp( \rho)$. The second condition has to be interpreted as an ``asymptotic'' boundary condition, see also \eqref{eq52} or Remark \ref{R1} below.  Notice that $\grad \psi$ may be non-zero on $\partial \supp(\rho)$, so that the boundary is ``free'' fo move.  For all other tangent fields, we set $\g_{\rho}(\delta\rho,\delta \rho)=+\infty$. 
Actually, arguing even more formally, Otto set the elliptic problem \eqref{eq16} in the entire space $\R^N$, cf.\ \cite[eq.\ (8)]{Otto01}. However, in order to motivate our later analysis, here, we equip the equation with its natural boundary conditions. That is, assuming enough regularity of $\rho$ and $\delta \rho$, the distributional solution $\psi$ of $-\div\left(\rho\grad \psi\right)=\delta\rho$ satisfies \eqref{eq16}. Equivalently, we can characterize the metric tensor variationally:
\[
\frac12 \g_{\rho}(\delta\rho, \delta\rho) \;=\; \sup_{\varphi} \left\{-\frac12\int \rho |\grad\varphi|^2\, dx- \int \varphi\delta\rho \, dx\right\},
\]
cf.\ \cite[eq.\ (2.9)]{OttoWestdickenberg05}, where the supremum is taken over all smooth functions $\varphi$ on $\R^N$. Indeed, $\g_{\rho}(\delta \rho,\delta\rho)$ is finite if and only if $\spt(\delta\rho)\subset \spt(\rho)$, and then the optimal $\varphi$ in this formulation satisfies \eqref{eq16}. When studying the gradient flow dynamics \eqref{eq31a}, it is enough to restrict the tangent plane to those tangent fields for which the metric tensor is finite. Hence, by identifying tangent fields $\delta\rho$ with the variables $\psi$ via \eqref{eq16}, the tangent plane at $\rho$ can be written as the weighted Sobolev space
\begin{equation}\label{eq50}
T_{\rho}\M\;=\; \left\{\psi\: :\: \int\rho|\grad\psi|^2\, dx<\infty\right\}.
\end{equation}

\medskip

The energy functional in the gradient flow interpretation is given by the entropy
\[
E(\rho)\;=\; \frac1{m-1}\int \rho^{m}\, dx.
\]
A short computation now shows the porous medium equation \eqref{eq1} is indeed the gradient flow for $(\M, \g)$ and $E$ defined as above, see \cite[Sec.\ 1.3]{Otto01} for details. 

\medskip

As for the confined \pme\ \eqref{eq2}, we remark that $\hat \rho$ still evolves according to the gradient flow on the same Riemannian manifold as $\rho$, now with total mass $\hat M$ and denoted by $(\hat\M,\hat\g)$, but for the energy $\hat E$ defined in \eqref{eq41}.

\medskip

We finally address the induced geodesic distance on the Riemannian manifold $(\hat \M,\hat\g)$. {\sc Benamou \& Brenier} \cite{BenamouBrenier00} discovered the relation between the Kantorovich mass transfer problem and continuum mechanics by identifying the geodesic distance on $(\hat \M,\hat\g)$ between two configurations $\hat\rho_0$, $\hat\rho_1$ having finite second moments $\int| \hat x|^2\hat \rho_i\, d\hat x<\infty$ with the {\em Wasserstein distance}
\[
d_2(\hat\rho_1,\hat\rho_0)^2\;=\; \inf_{\mu\in\Gamma(\hat\rho_1,\hat\rho_0)}\iint | \hat x_1-\hat x_0|^2\, \mu(d \hat x_1d\hat x_0).
\]
Here the infimum is taken over the space $\Gamma(\hat\rho_1,\hat\rho_0)$ of all nonnegative measures $\mu$ on $\R^N\times\R^N$ having marginals $\hat\rho_1\, d\hat x_1$ and $\hat\rho_0\, d\hat x_0$, i.e.,
\[
\iint \zeta(\hat x_i)\mu(d\hat x_1 d\hat x_0)=\int \zeta(\hat x_i)\  \hat\rho_i(\hat x_i) d\hat x_i
\]
for all $\zeta\in C_0(\R^N)$ and $i=0,1$. For a detailed study of Wasserstein distances and more general optimal transportation distances, we refer the interested reader to {\sc Villani's} two monographs \cite{Villani03,Villani09}.

\subsection{Linearization}\label{S1.3}

Instead of linearizing the confined \pme\ \eqref{eq2} around the Barenblatt profile $\hat \rho_{*}$, we imitate the approach of {\sc Denzler \& McCann} \cite{DenzlerMcCann05} and compute the Hessian of the entropy $\hat E$ at $\hat\rho_*$. In view of the gradient flow formulation described in the previous subsection, this is formally equivalent to linearizing \eqref{eq2} directly, but it has the advantage that a natural class of perturbations is intrinsically given by the tangent fields.

\medskip

Tangent fields $\delta \hat \rho$ are mass-preserving and concentrate on the support of $\hat \rho_{*}$. Following Otto, and as described in the previous subsection, we identify such tangent fields with new variables $\psi$ via the elliptic problem \eqref{eq16}, i.e.,
\begin{equation}\label{eq17}
\begin{array}{rcll}
-\div(\hat \rho_{*} \grad\psi)&=& \delta\hat\rho&\mbox{in }B_1,\\
\hat\rho_*\hat \grad\psi\cdot\nu&=&0&\mbox{on }\partial B_1,
\end{array}
\end{equation}
where $B_1$ denotes the ball of radius one centered at the origin, and thus $B_1=\supp(\hat\rho_*)$. Moreover, $\psi$ is such that
\begin{equation}\label{eq4}
\int \hat \rho_*|\hat  \grad\psi|^2\, d\hat x<\infty,
\end{equation}
cf.\ \eqref{eq50}. Geodesic curves $\{\hat\rho_s\}_{|s|\ll1}$ on $\M$ passing through $\hat\rho_*$ and pointing in the direction $\delta\hat\rho$ are obtained by {\sc McCann}'s \cite{McCann97} displacement interpolation
\begin{equation}\label{eq61}
\hat\rho_{*}(\hat x)\;=\; \det\left(I + s\hat \grad^2 \psi(\hat x)\right)\hat \rho_s\left(\hat x + s\hat \grad\psi(\hat x)\right)
\end{equation}
(see also \cite[Sec.\ 4.3]{Otto01}), that is, $\hat\rho_*$ is pushed forward by the map $\mbox{\rm id}  + s\hat \grad\psi(\tacka)$. We verify that
\[
\delta\hat \rho\;=\;\left.\partial_{s}\right|_{s=0}\hat \rho_s\;=\; -\hat\grad\cdot\left(\hat\rho_*\hat\grad\psi\right).
\]
Recalling that the Hessian $\hess \hat E$ of a function $\hat E$ on a Riemannian manifold $(\hat\M,\hat\g)$ can be calculated as the second derivative of $\hat E$ along geodesics, we define and have
\[
\Hess \hat E(\hat\rho_*)(\psi,\psi) \;:=\; \hat\g_{\hat\rho_*}(\delta \hat\rho, \hess \hat E(\hat\rho_*)\delta\hat\rho)\;=\;\left. \frac{d^2}{ds^2} \right|_{s=0}\hat E(\hat \rho_s).
\]
A {\em formal} computation yields:
\begin{eqnarray*}
\lefteqn{\Hess\hat  E(\hat\rho_*)(\psi,\psi)}\\
&=& m \int \hat \rho_*^{m-2} \left(\left.\partial_{s}\right|_{s=0}\hat \rho_s\right)^2\, d\hat x + \int \left(\frac{m}{m-1} \hat \rho_*^{m-1} + \frac{1}2 |\hat x|^2\right)\left.\partial_{s}^2\right|_{s=0}\hat \rho_s\, d\hat x.
\end{eqnarray*}
Observe that $\supp( \left.\partial_{s}^2\right|_{s=0}\hat \rho_s)\subset \supp(\hat\rho_*)$ and $\int \left.\partial_{s}^2\right|_{s=0}\hat \rho_s\, d\hat x=0$, and so the second integral on the right vanishes thanks to the definition of the Barenblatt profile $\hat \rho_*$. Because of $\left.\partial_{s}\right|_{s=0}\hat \rho_s=-\hat\grad\cdot\left(\hat\rho_*\hat\grad\psi\right)$, we can rewrite the Hessian as
\[
\Hess \hat E(\hat \rho_*)(\psi,\psi)\; =\; m \int \hat \rho_*^{m-2}\left(\hat \grad\cdot \left(\hat \rho_* \hat \grad\psi\right)\right)^2\, d\hat x .
\]
Finally, integration by parts yields
\[
\Hess\hat  E(\hat \rho_*)(\psi,\psi)\; =\; \int \hat \rho_* \hat \grad \psi \cdot \hat \grad\left(-m\hat \rho_*^{m-2}\hat \grad\cdot\left(\hat \rho_*\hat \grad\psi\right)\right)\, d\hat x.
\]
At this point, we consider the derivation of the Hessian on a purely formal level. The above computations certainly hold true for functions $\psi$ that are smooth up to the boundary. In this case, the integrals are well defined and we can integrate by parts since $\hat\rho_*^{m-1}$ vanishes on $\partial \supp(\hat\rho_*)$.

\medskip

We finally complete this subsection by introducing the {\em displacement Hessian}
\begin{equation}\label{eq6}
\Ha \psi\;=\;  -m\hat \rho_*^{m-2}\hat \grad\cdot\left(\hat \rho_*\hat \grad\psi\right).
\end{equation}
In view of the explicit formula \eqref{eq51} for $\hat \rho_*$, the displacement Hessian can be rewritten as
\begin{equation}\label{eq5}
\Ha\psi(x)\;=\; 
  - \frac{m-1}2\left(1 - |\hat x|^2 \right)\hat \laplace \psi(\hat x)+\hat x\cdot\hat \grad\psi(\hat x)\quad\mbox{for }\hat x\in B_1.
\end{equation}
In terms of $\Ha$, the linearized confined \pme\ reads $\partial_{\hat t} \psi +\Ha \psi=0$, or on the level of $\delta \hat \rho$:
\begin{equation}\label{eq60}
\partial_{\hat t} \delta\hat \rho + \L^{-1}\Ha\L \delta\hat \rho \;=\;0,
\end{equation}
where the operator $\L$ is defined via $\L\delta\hat \rho=\psi$ and $\delta \hat \rho$ and $\psi$ are related in the usual way. Moreover, a formal analysis yields that $\Ha$ and $\hess$ are similar in the sense that $\hess=\L^{-1}\Ha\L$.

\medskip

The remainder of the paper is exclusively devoted to the analysis of the displacement Hessian. We will show in Proposition \ref{P1} below, that it defines a self-adjoint operator on a suitable domain $\D(\Ha)$. Our main result, stated in Theorem \ref{T1} below, gives the complete spectrum of $\Ha$.


\section{Rigorous part}
From now on we claim mathematical rigor. As we focus exclusively on the confined porous medium equation \eqref{eq2} we may simplify our notation and drop the hats from the rescaled variables.

\medskip

Let $H_{\rho_*}^1$ denote the class of all locally integrable functions on $B_1=\supp( \rho_* )$ such that \eqref{eq4} holds, i.e.,
\[
\|\psi\|_{H_{\rho_*}^1}^2:=\int\rho_*|\grad\psi|^2\, dx<\infty,
\]
with the identification of two functions if they only differ by a constant. Observe that this space is a separable Hilbert space for the topology induced by the norm $\|\cdot\|_{\H}$. In the appendix, we study some properties related to elliptic theory in $\H$.

\medskip

So far, the derivation of the Hessian and the definition of the displacement Hessian were only formally justified. Hence, as a first step before embarking into spectral analysis, we have to build a rigorous basis for our investigations. Since the formal computations of the preceding section are certainly valid for functions that are smooth up to the boundary of the support of the Barenblatt profile, i.e., $C^{\infty}(\bar B_1)$ functions, let us denote the displacement Hessian defined in \eqref{eq6} by $\left.\Ha\right|_{C^{\infty}(\bar B_1)}$. A short computation yields that this operator is nonnegative and symmetric (cf.\ proof of Proposition \ref{P1}). In Subsection \ref{S2.1} below, we show that $\HaC$ extends naturally to a nonnegative self-adjoint operator $\Ha$ with domain
\begin{eqnarray}
\lefteqn{\D(\Ha)=\left\{\psi\in H_{loc}^3(B_1)\cap H_{\rho_*}^1:\right.}\nonumber\\
&&\hspace{3em}\left.\rho^{m-2}_*\div\left(\rho_*\grad\psi\right)\in\H,\  \rho_*\grad\psi\cdot\nu=0\mbox{ on }\partial B_1 \right\}.\label{eq62}
\end{eqnarray}
Here, the condition $\rho_*\grad\psi\cdot\nu=0$ on $\partial B_1$ means that
\begin{equation}\label{eq52}
\int \rho_* \grad \xi \cdot\grad\psi\, dx\;=\; -\int \xi \div(\rho_* \grad\psi)\, dx \quad\mbox{for all }\xi \in H_{\rho_*}^1.
\end{equation}
We also remark that the regularity that is assumed in the definition of $\D(\Ha)$ implies that $\Ha\psi=-m\rho^{m-2}_*\div\left(\rho_*\grad\psi\right)$ for all $\psi\in\D(\Ha)$, where the derivatives have to be understood in the weak Sobolev sense.

\begin{remark}\label{R1}
The boundary conditions in \eqref{eq52} are {\rm asymptotic} boundary conditions in the sense that
\begin{equation}\label{eq58}
\lim_{|x|\uparrow1} \rho_*(x) \grad\psi(x)\cdot\frac{x}{|x|} = 0,
\end{equation}
for every function $\psi\in C^{\infty}( B_1)$. In fact, thanks to the density of $C^{\infty}(\bar B_1)$ functions in $\H$ by Lemma \ref{L1} in the appendix, and the Sobolev embedding $\H \subset L^2(B_1, \rho_*^{2-m} dx)$ in Lemma \ref{L4} in the appendix, it is enough to restrict condition \eqref{eq52} to functions $\xi\in C^{\infty}(\bar B_1)$. Then, the equivalence of \eqref{eq52} and \eqref{eq58} follows immediately (for instance, by an indirect argument).
\end{remark}

\subsection{Main results}

We are now in the position to state our main results:

\begin{theorem}\label{T1}
The operator $\Ha: \D(\Ha)\to \H$ is self-adjoint. Its spectrum consists only of eigenvalues, given by
\[
\lambda_{\ell k} \;=\;\ell + 2k + 2k(k + \ell + \frac N 2-1)(m-1),
\]
where $(\ell,k)\in\N_0\times\N_0\setminus\{(0,0)\}$ if $N\ge 2$ and $(\ell,k)\in\{0,1\}\times\N_0\setminus\{(0,0)\}$ if $N=1$. The corresponding eigenfunctions are given by polynomials of the form
\[
\psi_{\ell n k}(x) =   F(-k, \frac1{m-1} +\ell +\frac{N}2 -1 +k; \ell+\frac{N}2;|x|^2)Y_{\ell n}\left(\frac{x}{|x|}\right)|x|^{\ell},
\]
where $n\in\{1,\dots,N_{\ell}\}$ with $N_{\ell}=1$ if $\ell=0$ or $\ell=N=1$ and $N_{\ell}= \frac{(N+\ell-3)!(N+2\ell-2)}{\ell!(N-2)!}$ else. Here, $F(a,b;c;z)$ is a hypergeometric function and in the case $N\ge2$, $Y_{\ell n}$ is a spherical harmonic corresponding to the eigenvalue $\ell(\ell + N-2)$ of $-\laplace_{\S^{N-1}}$ with multiplicity $N_{\ell}$. Otherwise, if $N=1$ it is $Y_{\ell 1}(\pm1)=(\pm1)^{\ell}$.
\end{theorem}

In the statement of the theorem, we use the notation $\N_0=\N\cup\{0\}$, and $\laplace_{\S^{N-1}}$ is the Laplace--Beltrami operator on the sphere $\S^{N-1}$. The eigenfunctions corresponding to the eigenvalues $\lambda_{\ell k}$ are polynomials of degree $\ell +2k$, which are harmonic and homogeneous if $k=0$. In fact, the hypergeometric functions $F(a,b;c;z)$ in the statement of the theorem reduce to polynomials of degree $k$ in the case $a=-k$, see Subsection \ref{S2.3} below. Spherical harmonics $Y_{\ell n}$ are harmonic homogeneous polynomials of degree $\ell$. These functions are ubiquitous in mathematical physics, see e.g.\ \cite{Groemer96} to name but a single reference. Finally, in the case $N=1$, $Y_{01}$ and $Y_{11}$ are the eigenfunctions of the parity operator, cf.\ Subsection \ref{S2.2a}.

\medskip

We observe that the eigenvalues of the displacement Hessian $\Ha$ are affine functions of the parameter $m$. There are countably many constant levels $\lambda_{\ell 0}=\ell$. Above $\lambda_{10}$ and $\lambda_{20}$, the spectrum features a crossing of eigenvalues when varying $m$. A first level crossing occurs for the eigenvalues $\lambda_{30}=3$ and $\lambda_{01}= 2+ N(m-1)$ at $N(m-1)=1$, and the number of crossings increases with $\ell$. The eigenvalues coincide with the eigenvalues found by {\sc Denzler \& McCann} \cite{DenzlerMcCann05} for the fast diffusion regime in the sense that each eigenvalue continues to the range $m<1$ if $m$ is sufficiently close to $1$. For smaller $m$, the eigenvalues dissolve into continuous spectrum (that disappears in the limit $m\uparrow 1$). The occurrence of a continuous spectrum for the fast diffusion equation is analytically related to the fact that in the fast-diffusion regime, the number of moments possessed by the Barenblatt profile is finite --- in contrast to the situation for the porous medium equation. We can rule out the appearance of continuous spectrum by proving that the resolvent $\Ha^{-1}$ is compact, see Proposition \ref{P3}. In the limit $m\downarrow1$, we recover the spectrum of the Ornstein--Uhlenbeck operator $-\laplace + x\cdot \grad $. This coincidence is not surprising since the Ornstein--Uhlenbeck operator is the limit operator of $\Ha$ as $m\downarrow 1$. In this case, $\rho_*$ is a Gaussian.

\medskip

The knowledge of the spectrum of the displacement Hessian $\Ha$ suggests the asymptotic expansion to leading order of solutions $\rho=\rho(t)$ to the confined porous medium equation close to the stationary solution $\rho_*$. Arguing purely formally again, we recall the operator $\L$ defined by $\L\delta \rho=\psi$ where $\delta\rho$ and $\psi$ are related in the usual way \eqref{eq17}, cf.\ Subsection \ref{S1.3}, and observe that the spectrum of  $\Ha$ coincides with the spectrum of $\L^{-1}\Ha\L$. Moreover, the corresponding eigenfunctions $\delta\rho_{\ell kn}=\L^{-1}\psi_{\ell kn}$ form an orthogonal basis of the Hilbert space $L^2(B_1,\rho_*^{m-2}dx)$. For solutions $\rho(t)$ close to $\rho_*$, equation \eqref{eq2} reads to leading order
\[
\partial_t (\rho(t)-\rho_*)\;\approx\; -\L^{-1}\Ha\L (\rho(t)-\rho_*) ,
\]
cf.\ \eqref{eq60}, and thus, exploiting the knowledge of the complete spectrum of $\Ha$ (and thus of $\L^{-1}\Ha\L$)
\begin{eqnarray*}
\lefteqn{\rho(t)-\rho_*}\\
&=& \sum_{\ell kn} c_{\ell k n} \delta \rho_{\ell k n} e^{- \lambda_{\ell k}t} + \sum_{\ell kn}\sum_{\ell' k'n'} c_{\ell k n, \ell' k'n'} \delta \rho_{\ell k n}\delta\rho_{\ell' k' n'} e^{- (\lambda_{\ell k} + \lambda_{\ell'k'})t}+ \dots,
\end{eqnarray*}
where the $c$'s are constants that depend on the initial data $\rho(0)$ only. In order to investigate the precise asymptotics of the nonlinear equation \eqref{eq2}, one has to take into account terms of higher order in $\rho(t)-\rho_*$. This problem will be addressed in future research. In this context we should also mention the work of {\sc Koch} \cite[Ch.\ 5.3.4]{Koch}, which offers a framework which seems to be suitable for a rigorous treatment of higher order asymptotics. For the one-dimensional porous medium equation, the full asymptotics was rigorously computed by {\sc Angenent} in \cite{Angenent88}, based on an earlier investigation by {\sc Zel'dovic \& Barenblatt} \cite{ZeldovicBarenblatt58}. This present work may be considered as a first attempt to complement Angenent's result in the multivariable case. Recently, {\sc Denzler, Koch \& McCann} studied the higher order asymptotic behavior of solutions to the fast-diffusion equation \cite{DenzlerKochMcCann13}, using subtle dynamical systems arguments in weighted H\"older spaces. The authors develop a rigorous theory analogous to the one in \cite{Koch} to translate the knowledge of the spectrum of the (formally) linearized equation computed by {\sc Denzler \& McCann} \cite{DenzlerMcCann05} into information on the higher order asymptotics.

\medskip

The main purpose of understanding the full asymptotic expansion of solutions in terms of the eigenfunctions of the Hessian around the self-similar solution is to get a deeper understanding of the long-time behavior of solutions. Although we have to keep in mind that our investigation is not fully rigorous in the derivation of the displacement Hessian operator, we will in the following discuss the the role of the smallest eigenvalues and their corresponding eigenfunctions for the convergence towards the attractor $\rho_*$ in dimensions $N\ge2$.  We recall that the geodesic curves in the Wasserstein space passing through the Barenblatt solution $\rho_*$ along the vector field $\grad \psi$ are characterized  by $\rho_*(x)=\det\left(I +s\grad^2\psi\right)\rho_s(x+s\grad\psi)$ (cf.\ \eqref{eq61}). With this formula at hand, the interpretation of the eigenfunctions as generators of mass transport is immediate. The smallest eigenvalue $\lambda_{10}=1$ corresponds to transitions along the axis $e_n$ with $n\in\{1,\dots,N_1=N\}$ (see Figure \ref{fig}a). Indeed, in this case, the eigenfunction is a homogeneous polynomial of degree one, and the displacement is generated by affine perturbations of the identity map $x+sc_n e_n$ (here, $c_n$ denotes a normalizing factor). This fact is in good agreement with the optimal convergence rate of solutions of the porous medium equation to the Barenblatt profile,
\begin{figure}[t]
\begin{center}(a)\scalebox{.6}{\includegraphics{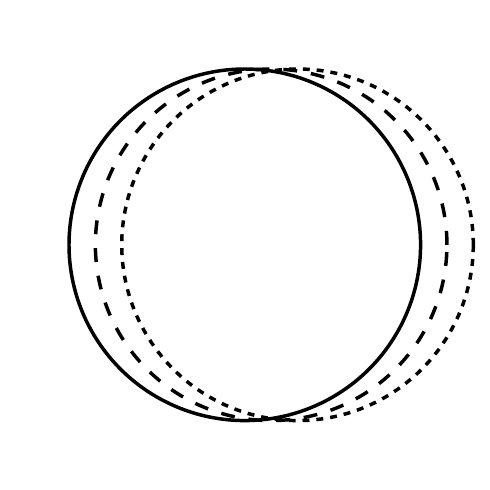}}
	(b)\scalebox{.6}{\includegraphics{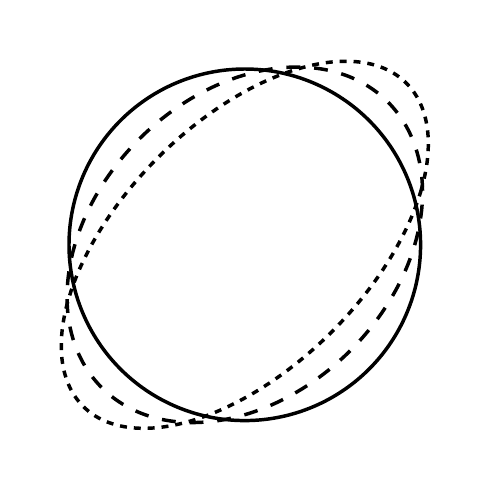}}
	(c)\scalebox{.6}{\includegraphics{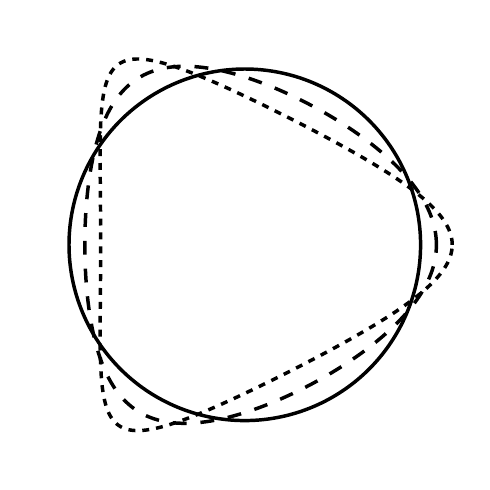}}
	(d)\scalebox{.6}{\includegraphics{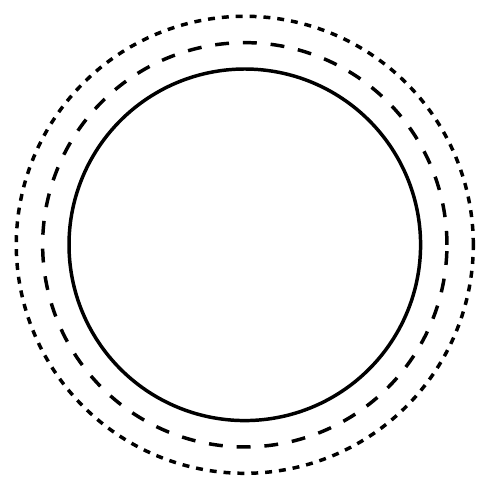}}
\end{center} \caption{2D plots of the supports of the affine transformations (dashed and dotted lines) of the Barenblatt profile (solid lines) generated by the eigenfunctions corresponding to the eigenvalues $\lambda_{10}$, $\lambda_{20}$, $\lambda_{30}$, and $\lambda_{01}$ (from the left to the right).\label{fig}} 
\end{figure}
\[
d_2(\rho(t), \rho_*)\le e^{-t}d_2(\rho_0,\rho_*),
\]
here stated in the Wasserstein topology \cite[Theorem 1]{Otto01}. In fact, spatially translated Barenblatt profiles saturate this bound. The eigenfunctions to the second smallest eigenvalue $\lambda_{20} =2$ generate affine transformations on $\R^N$ modulo rotations (Figure \ref{fig}b). To see this, we observe that $\psi_{20n}$ is a harmonic homogeneous polynomial of degree two, and thus $x+s\grad \psi=x+s A_{2n}x$ for some symmetric and trace-free matrix $A_{2n}$. In other words, solutions to \eqref{eq2} generated by affine transformations show a rate of convergence of precise order $e^{-2t}$ to the Barenblatt profile. {\sc Denzler \& McCann} \cite{DenzlerMcCann08} give an exact description of the invariant manifold corresponding to these affine perturbations. The role of the next term in the expansion depends on the value of $(m-1)N$. As indicated above, there is a first level crossing of the eigenvalues $\lambda_{\ell0}$ and $\lambda_{01}$: It is $\lambda_{01}\le\lambda_{\ell0}$ precisely for $\ell\ge 1/\alpha$. Eigenfunctions to the eigenvalue $\lambda_{01}$ correspond to isotropic dilations of the Barenblatt profile (see Figure \ref{fig}d). Indeed, the spherical harmonic $Y_{01}$ degenerates to a constant and $\psi_{011}$ is proportional to $1-(\alpha (m-1)N )^{-1}|x|^2$. It follows that $x+s\grad\psi =(1-cs)x$ for some constant $c$. Such solutions converge with a rate $e^{-t/\alpha}$. The mass transport corresponding to the $(\ell,0)$-modes with $\ell\ge3$ is less obvious,  as the corresponding eigenfunctions $\psi_{\ell0n}$ are homogenous harmonic polynomials of degree $\ell$. We only discuss the situation where $\ell=3$. In this case, the transformation is quadratic and leads to triangular deformations (see Figure \ref{fig}c). The rate of convergence is $e^{-3t}$. The deformations corresponding to the remaining eigenvalues can be similarly computed although the complexity of the underlying symmetries is increasing.

\medskip

We finally remark that the knowledge of spectral properties of the displacement Hessian $\Ha$, and of the smallest eigenvalue $\lambda_{10}=1$ in particular, immediately yields the sharp spectral gap estimate
\[
\int \rho_* |\grad \psi|^2\, dx\;\le\; \Hess E(\rho_*)(\psi,\psi),
\]
for all $\psi \in\H$. This estimate, was already derived by {\sc Otto} \cite[Sec.\ 4.4]{Otto01} and builds on {\sc McCann}'s displacement convexity \cite{McCann97}.

\medskip

We prove Theorem \ref{T1} in the remainder of this article: In Subsection \ref{S2.1} we show that the displacement Hessian $\Ha$ can be understood as a self-adjoint operator on $\H$ (Proposition \ref{P1}). Moreover, we show that its spectrum is purely discrete (Proposition \ref{P3}), and thus, the computation of the spectrum reduces to the eigenvalue problem. This requires some preparation. In Subsection \ref{S2.2}, we treat the case $N\ge2$ and make a change of variables into spherical coordinates, which has the advantage that the eigenvalue problem reduces to a one-dimensional problem for the radial components of the displacement Hessian. The $N=1$ case is considered in Subsection \ref{S2.2a}, where we split the problem by parity. The resulting eigenvalue problems are solved in Subsection \ref{S2.3} for all $N$ simultaneously.

\subsection{Self-adjointness of $\Ha$ and discreteness of the spectrum}\label{S2.1}
Since $\HaC$ is nonnegative and symmetric (cf.\ proof of Proposition \ref{P1}) and because $C^{\infty}(\bar B_1)$ is densely contained in $\H$ by Lemma \ref{L1} in the appendix, the operator $\HaC$ is closable in $\H$ and its closure is still nonnegative and symmetric. We denote this closure again by $\Ha$. Using the embedding from Lemma \ref{L4} in the appendix, one can easily verify that the domain $\D(\Ha)$ of the closed operator is given by \eqref{eq62}. In the following we show that $\Ha$ is self-adjoint and that its spectrum is purely discrete, i.e., it consists only of eigenvalues with finite multiplicity and which do not have a finite accumulation point.  For the theory of unbounded operators in general and self-adjointness in particular we refer to \cite[Ch.\ 13]{Rudin} or \cite{Schmudgen}.

\medskip

The following proposition contains a first statement on the spectrum of $\Ha$, namely that the spectrum is real and contained in the ray $(0,\infty)$. The first fact is a consequence of self-adjointness, the second one of nonnegativity together with the fact that $\Ha$ has a bounded inverse (and thus $0$ must be in the resolvent set).

\begin{proposition}[Self-adjointness]\label{P1}
The operator $\Ha: \D(\Ha)\to \H$ is nonnegative, self-adjoint, and has a bounded inverse.
\end{proposition}

The next result shows that the spectrum consists of eigenvalues only.

\begin{proposition}[Discrete spectrum]\label{P3}
The operator $\Ha: \D(\Ha)\to\H$ has a purely discrete spectrum. 
\end{proposition}

We start with the
\begin{proof}[Proof of Proposition \ref{P1}.] {\em Step 1. $\Ha$ is densely defined, nonnegative, and symmetric.} Since $C^{\infty}(\bar B_1)$ is dense in $\H$ by Lemma \ref{L1} in the appendix and $C^{\infty}(\bar B_1)\subset\D(\Ha)\subset\H$, we immediately have that $\Ha$ is densely defined. By the definition of $\Ha$ and by integration by parts we have
\[
\int \rho_* \grad \varphi \cdot \grad\left(\Ha\psi\right)\, dx\;=\; m\int \rho_*^{m-2} \div(\rho_* \grad\varphi)\div(\rho_*\grad\psi)\, dx
\]
for all $\varphi, \psi\in \D(\Ha)$. Notice that the boundary term vanishes when integrating by parts because of \eqref{eq52} and $\Ha\psi\in\H$. Now nonnegativity follows from choosing $\varphi=\psi$ and symmetry comes from the symmetry in $\varphi$ and $\psi$ of the right hand side of the above identity.

\medskip

{\em Step 2. $\Ha$ is onto.} This is a consequence of Lemmas \ref{L4} \& \ref{L5} in the appendix. Indeed, given $u\in \H$, we may without loss of generality assume that $\int \rho_*^{2-m}u\, dx=0$ as we identify functions that are the same up to a constant. By Lemma \ref{L4} we have that $u\in L^2(B_1, \rho_*^{2-m}dx)$. We define $\tilde u=\frac1m\rho_*^{2-m}u$ and observe that this function satisfies the hypotheses of Lemma \ref{L5}. Consequently, there exists a function $\psi\in H_{\loc}^2\cap\H$ such that
\begin{eqnarray*}
-\div (\rho_* \grad\psi) &= &\tilde u\quad\mbox{in }B_1,\\
\rho_* \grad \psi\cdot\nu&= &0\quad\mbox{on }\partial B_1.
\end{eqnarray*}
By the definition of $\tilde u$, the equation in the ball can be rewritten as $-m\rho_*^{m-2}\div\left(\rho_*\grad\psi\right)=u$, i.e., $\Ha \psi=u$. Higher interior regularity estimates (cf.\ \cite[Theorem 8.12]{Rudin}) moreover show that $\psi\in H_{\loc}^3\cap\H$. Hence $\psi\in\D(\Ha)$. This proves that $\Ha$ is onto.

\medskip

{\em Step 3. Conclusion.} We conclude with the help of two basic facts from abstract operator theory, namely: A densely defined symmetric operator is: (A) one-to-one if its range is dense, cf.\ \cite[Theorem 13.11(c)]{Rudin}; (B) self-adjoint and invertible with bounded inverse if it is onto, cf.\ \cite[Theorem 13.11(d)]{Rudin}. Thanks to Steps 1 and 2, $\Ha$ satisfies the assumptions of both (A) and (B), and thus, $\Ha$ is non-negative, self-adjoint, and has a bounded inverse. This completes the proof of Proposition \ref{P1}.
\end{proof}

We finally prove Proposition \ref{P3}.
\begin{proof}[Proof of Proposition \ref{P3}.] Notice that by Proposition \ref{P1}, $\Ha$ is invertible and its inverse is bounded. In particular, $0$ is in the resolvent set of $\Ha$. Hence, in order to prove that $\Ha$ has a purely discrete spectrum, it is enough to show that the resolvent $\Ha^{-1}: \H\to\H$ is compact, cf.\ \cite[Proposition 2.11]{Schmudgen}.

\medskip

In the definition of the Hilbert space $\H$, we identified functions that only differ by a constant. In the following, for notational convenience, we fix constants by additionally requiring that
\[
\int\rho_*^{2-m}\varphi\, dx \;=\;0\quad\mbox{for all }\varphi\in \H.
\]
Then $\Ha^{-1}$ is defined as follows: For any $\xi\in \H$, we have that $\psi:=\Ha^{-1}\xi$ solves
\begin{eqnarray*}
-m\rho_*^{m-2}\div (\rho_* \grad\psi) &= &\xi\quad\mbox{in }B_1,\\
\rho_* \grad \psi\cdot\nu&= &0\quad\mbox{on }\partial B_1.
\end{eqnarray*}


The argument which shows that $\Ha^{-1}$ is a compact operator is standard, and is based on a Rellich lemma, namely Lemma \ref{L7} in the appendix. We display the argument since we are dealing with weighted (and thus non-standard) Sobolev norms. Let $\{\xi_n\}_{n\in\N}$ denote a bounded sequence in $\H$. Since $\Ha^{-1}$ is a bounded operator, also the sequence $\{\psi_n\}_{n\in\N}$ with $\Ha^{-1}\xi_n=\psi_n$ is bounded. Hence, because $\H$ is a separable Hilbert space, there exist $\xi$ and $\psi$ in $\H$ and subsequences (that we do not relabel) such that $\{\xi_n\}_{n\in\N}$ and $\{\psi_n\}_{n\in\N}$ converge weakly to $\xi$ and $\psi$, respectively. Moreover, by Lemma \ref{L7} in the appendix, this convergence is strong in $L^2(B_1, \rho_*^{2-m}dx)$. We need to show that $\{\psi_n\}_{n\in\N}$ converges to $\psi$ strongly in $\H$. Integrating by parts and using the definition of $\psi_n$, we have
\[
\int \rho_*\grad\psi_n\cdot\grad\varphi\, dx \;=\;\frac1m\int\rho_*^{2-m} \xi_n\varphi\,dx,
\]
for all $\varphi\in\H$, and thus passing to the limit $n\uparrow\infty$, we see that $\psi\in\D(\Ha)$ and $\Ha^{-1}\xi=\psi$. Now, we observe by the strong convergence in $L^2(B_1,\rho_*^{2-m}dx)$ that
\[
\lim_{n\uparrow\infty} \int \rho_*|\grad\psi_n|^2\, dx\;=\; \lim_{n\uparrow\infty} \frac1m\int\rho_*^{2-m} \psi_n\xi_n\,dx \;=\;  \frac1m\int\rho_*^{2-m} \psi\xi\,dx\;=\; \int \rho_*|\grad\psi|^2\, dx,
\]
and the last identity follows from integrating by parts again and using $\Ha^{-1}\xi=\psi$ and $\psi\in\D(\Ha)$. This shows that the $\H$ norms of the sequence $\{\psi_n\}_{n\in\N}$ converge, and together with weak convergence, this yields strong convergence of $\{\psi_n\}_{n\in\N}$ (or equivalently of $\{\Ha^{-1}\xi_n\}_{n\in\N}$) in $\H$. We deduce that $\Ha^{-1}$ is compact.
\end{proof}

\subsection{Case $N\ge 2$: Separation of variables in spherical coordinates}\label{S2.2}

In this subsection, we most widely summarize Subsection 4.1 of \cite{DenzlerMcCann05}. As our argumentation for the \pme\ follows closely the one for the fast-diffusion equation by {\sc Denzler \&\ McCann}, we entirely skip proofs and refer to \cite{DenzlerMcCann05} for details. An important step in the work of Denzler and McCann is the change of perspective in the spectral analysis of the displacement Hessian $\Ha$ which comes along with the transformation of the operator into spherical coordinates. In fact, the change of variables into spherical coordinates is motivated by the crucial insight that the displacement Hessian $\Ha$ and the spherical Laplacian $\laplace_{\S^{N-1}}$ (i.e., the Laplace--Beltrami operator on the sphere $\S^{N-1}$) commute. This is basically a consequence of the fact that the evolution commutes with rotations. Consequently, both operators can be simultaneously diagonalized. In particular, as the spectrum of $\laplace_{\S^{N-1}}$ is well-known, our spectral analysis will reduce to the study of the radial part of $\Ha$, which amounts to a one-dimensional operator. We remark that this Ansatz is frequently used in the spectral analysis of Schr\"odinger operators, cf.\ \cite{NikiforovUvarov}.

\medskip

We transform $x\in\R^N$ into spherical coordinates, i.e., $x=r\omega$ with $(r,\omega)\in[0,\infty)\times\S^{N-1}$, and recall that under this transformation, the Laplacian reads
\[
\laplace_{\R^N}\;=\; \partial_r^2 + \frac{N-1}r \partial_r + \frac1{r^2}\laplace_{\S^{N-1}}.
\]
Here $\laplace_{\S^{N-1}}$ denotes the spherical Laplacian. Since the Barenblatt profile $\rho_*$ is a radially symmetric function, i.e., $\rho_*  = \rho_*(r)$, one readily checks that $ \Ha$ and $\laplace_{\S^{N-1}}$ commute, cf.\ \eqref{eq5}. 

\medskip

The eigenvalues of the spherical Laplacian $\laplace_{\S^{N-1}}$ are $\mu_{\ell} = \ell(\ell + N-2)$ with multiplicity $N_{\ell} = \frac{(N+\ell-3)!(N+2\ell-2)}{\ell!(N-2)!}$ where $N_0=1$. That is,
\[
-\laplace_{\S^{N-1}} Y_{\ell n}\;=\; \mu_{\ell} Y_{\ell n}\quad\mbox{for }\ell\in\N_0\mbox{ and }n\in\{1,\dots,N_{\ell}\},
\]
and the eigenfunctions $Y_{\ell n}$ are the spherical harmonics $Y_{\ell n}: \S^{N-1}\to\R$, which form a complete orthonormal basis for $L^2(\S^{N-1},d\omega)$:
\begin{equation}\label{eq7}
\int_{\S^{N-1}}Y_{\ell n}(\omega) Y_{\ell' n'}(\omega)\, d\omega\;=\; \delta_{\ell\ell'}\delta_{nn'}.
\end{equation}
For more details about the spectrum and the eigenfunctions of the spherical Laplacian, see also \cite[Ch.\ 3]{Groemer96}.

\medskip

Exploiting the orthonormality of the spherical harmonics, we observe that for every radially symmetric function $f: (0,1)\to \R$ the quantity $\|fY_{\ell n}\|_{H^1_{\rho_*}}$ is independent of the particular choice of $n\in\{1,\dots,N_{\ell}\}$,
\[
  \|fY_{\ell n}\|_{H_{\rho_*}^1}^2\;=\; \int_0^1 \left((f'(r))^2 + \frac{\mu_{\ell}}{r^2} (f(r))^2\right)\rho_*(r)r^{N-1}\, dr.
\]
This leads us to the definition of the norm $\| f\|_{H^1_{\ell}} = \|fY_{\ell n}\|_{H_{\rho_*}^1}$
and the corresponding Sobolev space
\[
H_{\ell}^1\;=\; \left\{ f\in H_{\loc}^1(0,1):\: \|f\|_{H^1_{\ell}}<\infty\right\}.
\]
In the case $\ell=0$ (i.e., $\mu_{\ell}=0$), we have to identify functions that only differ by a constant. In fact, $H_{\ell}^1$ is a Hilbert space and there is an isometric Hilbert space isomorphism
\[
H_{\rho_*}^1 = \bigoplus_{\ell=0}^{\infty} \bigoplus_{n=1}^{N_{\ell}} H_{\ell}^1
\]
given by
\begin{eqnarray*}
\psi (r\omega)&=& \sum_{\ell=0}^{\infty}\sum_{n=1}^{N_{\ell}} f_{\ell n}(r)Y_{\ell n}(\omega),\\
 f_{\ell n}(r)&=& \int_{\S^{N-1}}\psi(r\omega)Y_{\ell n}(\omega)\, d\omega,
\end{eqnarray*}
where the series are converging in $L^2(\S^{N-1},d\omega)$, and the isometry reads
\[
\|\psi\|_{H_{\rho_*}^1}^2 \;=\; \sum_{\ell=0}^{\infty} \sum_{n=1}^{N_{\ell}} \| f_{\ell n}\|_{H_{\ell}^1}^2.
\]

\medskip

The orthogonal decomposition of the Hilbert space $\H$ into eigenspaces generated by the spherical harmonics permits us to expand the displacement Hessian into a series of radially symmetric operators. More precisely, choosing $\psi\in\D(\Ha)$ with $\psi(r\omega)=f(r)Y_{\ell n}(\omega)$, we have $\Ha \psi = (\Ha_{\ell}f)Y_{\ell n}$, where $\Ha_{\ell}$ is the orthogonal projection of $\Ha$ onto $H_{\ell}^1$, namely
\begin{equation}\label{eq9}
\left( \Ha_{\ell} f\right)(r)\;=\; -m\rho_*(r)^{m-1}\left(f''(r) + \frac{N-1}r f'(r) - \frac{\mu_{\ell}}{r^2}f(r)\right) +  rf'(r).
\end{equation}
It is clear that this operator is again non-negative and self-adjoint with domain $\D(\Ha_{\ell}) = \{f\in H^1_{\ell}:\: fY_{\ell n}\in\D(\Ha)\}$. For an explicit definition of $\D(\Ha_{\ell})$, we have to investigate how the asymptotic boundary condition $\rho_* \grad\psi\cdot\nu=0$ behaves under the transformation into spherical coordinates. (This investigation is not necessary in the work of Denzler and McCann as they consider functions that are spread all over $\R^N$. In particular, in the fast-diffusion case the Barenblatt profile is positive everywhere.) The integrability condition on the free boundary $\partial B_1$ will become a selection criterion for identifying eigenfunctions of $\Ha_{\ell}$, and therefore, it has to be studied with care. For $\psi(r \omega)=f(r)Y_{\ell n}(\omega)$ and $\xi\in\H$, we first compute that
\[
-\int\div\left(\rho_*\grad\psi\right)\xi\, dx
\;=\;-\int_0^1 \left(\rho_*f'r^{N-1}\right)' g_{\ell n} \, dr + \mu_{\ell} \int_0^1 \rho_* f g_{\ell n} r^{N-3}\, dr,
\]
and
\[
\int \rho_*\grad\psi\cdot\grad\xi\, dx\;=\; \int_0^1 \rho_* \left(f' g_{\ell n}'  + \frac{\mu_{\ell}}{r^2} f g_{\ell n}\right) r^{N-1}\, dr,
\]
where $g_{\ell n}(r) = \int_{\S^{N-1}} \xi(r\omega) Y_{\ell n}(\omega)\, d\omega$. Consequently, \eqref{eq52} becomes
\begin{equation}\label{eq55}
\int_0^1 \rho_* f' g' r^{N-1}\, dr\; =\; - \int_0^1 (\rho_* f' r^{N-1})' g\, dr\quad\mbox{for all } g\in H^1_{\ell}.
\end{equation}
As the boundary term at $r=0$ automatically vanishes for every $f, g\in H_{\ell}^1$ with $\Ha_{\ell}f\in H^1_{\ell}$, we simply write $\left.\rho_* f'\right|_{r=1}=0$ to indicate that \eqref{eq55} holds. Now the domain of self-adjointness of $\Ha_{\ell}$ can be written as
\[
\D(\Ha_{\ell})\;=\; \left\{ f \in H_{loc}^3(0,1)\cap H_{\ell}^1:\:\Ha_{\ell} f \in H_{\ell}^1,\ \left.\rho_* f'\right|_{r=1}=0  \right\}.
\]

\begin{remark}\label{R2}
In analogy to the statement \eqref{eq58} in Remark \ref{R1}, we like to point out that the asymptotic boundary condition in the radially symmetric variables reduces to
\begin{equation}\label{40}
\lim_{r\uparrow 1}\rho_*(r)f'(r)\;=\;0,
\end{equation}
for every function $f\in C^{\infty}(0,1)$.
\end{remark}

\subsection{Case N=1: Symmetrization}\label{S2.2a}
In the case $N=1$, one easily checks that $\Ha$ commutes with the parity operator $\P\psi(x)=\psi(-x)$, and thus, both can be simultaneously diagonalized. The eigenfunctions of $\P$ are given by $Y_{\ell1}(\pm1)=(\pm1)^{\ell}$, for $\ell\in\{0,1\}$, and correspond to the eigenvalues $(-1)^{\ell}$. For notational convenience, we reuse the notation from the multidimensional situation presented in the previous subsection and denote by $H^1_{0}$ and $H^1_1$ the restriction of the Hilbert space $\H$ onto even and odd functions, respectively. With $r=|x|$ and $\omega =x/|x|$, we have the obvious isometric Hilbert space isomorphism $\H=H^1_0\oplus H^1_1$ given by
\begin{eqnarray*}
\psi(r\omega) & = & \sum_{\ell=0}^1 f_{\ell 1}(r) Y_{\ell1}(\omega),\\
f_{\ell1}(r) &=& \frac12\sum_{\omega\in\{\pm1\}} \psi(r\omega)Y_{\ell1}(\omega).
\end{eqnarray*}
We now write $\D(\Ha_{\ell})=\D(\Ha)\cap H^1_{\ell}$, and then $\Ha_{\ell}: \D(\Ha_{\ell})\to H^1_{\ell}$ is the restriction of $\Ha$ onto $H^{1}_{\ell}$, given by
\[
\Ha_{\ell}f(r) = -m\rho_*(r)^{m-1}f''(r) + rf'(r)\quad\mbox{for } f\in H^1_\ell.
\]
Notice that this formula is consistent with \eqref{eq9} because $\mu_{\ell}=0$. The asymptotic boundary condition is \eqref{eq55} with $N=1$. For symmetry reasons, it is enough to consider the norms on $H^1_{\ell}$ on the interval $(0,1)$, i.e.,
\[
\|f\|_{H^1_\ell}^2=2\int_0^1 (f'(r))^2\rho_*(r)\, dr.
\]
For $\ell=0$, we again identify functions that only differ by an additive constant.

\subsection{The eigenvalue problem for $\Ha_{\ell}$}\label{S2.3}
In this subsection, we solve the eigenvalue problem for the operators $\Ha_{\ell}: \D(\Ha_{\ell})\to H^1_{\ell}$, where $\ell\in\N_0$ if $N\ge 2$ and $\ell\in\{0,1\}$ if $N=1$.

\begin{proposition}[Eigenvalue problem for $\Ha_{\ell}$]\label{P2}
The eigenvalue problem
\[
\Ha_{\ell} f\;=\;\lambda f
\]
in $\D(\Ha_{\ell})$ has exactly the eigenvalues
\[
\lambda_{\ell k} \;=\;\ell + 2k + 2k(k + \ell + \frac N 2-1)(m-1),
\]
where $k\in\N_0$ such that $(\ell,k)\not=(0,0)$. The corresponding eigenfunctions are the polynomials
\[
f_{\ell k}(r)= r^{\ell} F(-k, \frac1{m-1} +\ell +\frac{N}2 -1 +k; \ell+\frac{N}2;r^2),
\]
where $F(-k,b;c;z) = 1+ \sum_{j=1}^k \frac{(-k)_j (b)_j}{(c)_j j!}z^j$ with $(s)_j= s(s+1)\dots(s+j-1)$.
\end{proposition}

\medskip

From the study of the operator $\Ha$ in Proposition \ref{P1} we already know that any eigenvalue of $\Ha_{\ell}$ must be real and positive, i.e., $\lambda>0$. Notice that in view of the explicit formulas \eqref{eq9} \& \eqref{eq51} for $\Ha_{\ell}$ and $\rho_*$, the equation $\Ha_{\ell} f=\lambda f$ reads
\begin{equation}\label{eq10}
f''  + \left(\frac{N-1}r - \frac2{m-1}\frac{r}{1-r^2}\right)f' + \left(\frac{2\lambda}{m-1}\frac1{1-r^2}-\frac{\mu_{\ell}}{r^2} \right)f
\;=\; 0 \quad\mbox{in }(0,1).
\end{equation}
This is a linear differential equation of second order with singularities at the endpoints of the interval $(0,1)$. It is of Fuchsian type, that means, all singular points, here $0$, $1$, and $\infty$, are regular. For every given $\lambda$, this differential equation has two linearly independent families of solutions. It is well known that the study of second-order Fuchsian ODEs with three regular singular points is intertwined with the study of hypergeometric functions $F(a,b;c;z)$ defined by
\begin{equation}\label{eq13}
F(a,b;c;z)\;=\; \sum_{j=0}^{\infty} \frac{(a)_j (b)_j}{(c)_j j!} z^j,
\end{equation}
where $a,b,c,z\in\R$ and $c$ is not a non-positive integer. Here, the notation involves the Pochhammer symbols (or extended factorials)
\[
(s)_j = s(s+1)\dots(s+j-1),\quad\mbox{for }j\ge1,\quad\mbox{and}\quad (s)_0=1.
\]
It is easily verified that the series converges if $|z|<1$. In the case that moreover $c>b>0$, the hypergeometric function $F(a,b;c;z)$ has the integral representation
\[
F(a,b;c;z)\;=\; \frac{\Gamma(c)}{\Gamma(b)\Gamma(c-b)}\int_0^{\infty} t^{b-1} (1+t)^{a-c} (1+t-zt)^{-a}\, dt,
\]
were $\Gamma(s)$ denotes Euler's Gamma function. We finally quote the fact that is most relevant for our purposes: For every choice of $a,b,c,z\in\R$, $c$ not a non-positive integer, the function $F(a,b;c;z)$ satisfies the {\em Hypergeometric differential equation}
\begin{equation}\label{eq11}
F'' + \left(\frac{c}{z} - \frac{a+b-c+1}{1-z}\right)F' - \frac{ab}{z(1-z)} F\;=\; 0\quad\mbox{in }(0,1),
\end{equation}
where the primes indicate derivatives with respect to $z$, i.e, $F' = \partial_z F$ and $F'' = \partial_z^2 F$. For detailed discussions of hypergeometric functions, we refer to \cite{Rainville1,Rainville2,BealsWong10}. A compact catalogue of the most important facts can be found in \cite{Abramowitz}.

\medskip

Now, we come back to the differential equation \eqref{eq10} and its relation to the hypergeometric equation \eqref{eq11}. Our goal is to transform \eqref{eq10} into the hypergeometric differential equation \eqref{eq11}, for which all solutions and their asymptotics at the singular points $0$, $1$, and $\infty$ are well-known. Making the Ansatz $f(r) = r^{\gamma} F(a,b;c;r^2)$ and supposing that $f$ is a solution, \eqref{eq10} transforms into
\begin{eqnarray}
\lefteqn{F'' + \left( \frac{2\gamma +N}2\frac1{z} - \frac1{m-1}\frac1{1-z}\right)F'}\nonumber\\
&&\mbox{}+\left(\frac{\gamma(\gamma + N-2) - \mu_{\ell}}4\frac1{z^2} - \frac{\gamma-\lambda}{2(m-1)}\frac1{z(1-z)}\right)F \;=\;0,\label{eq38}
\end{eqnarray}
with $z=r^2$. Comparing this differential equation with \eqref{eq11} and identifying $a$, $b$, $c$, and $d$ yields a first solution to \eqref{eq10}. A second linearly independent solution can be deduced from the well-known (but complex) theory of hypergeometric functions. We have the following

\begin{lemma}\label{L2}
Let
\begin{eqnarray}
a &=& \frac12\left(\frac1{m-1} + \ell + \frac{N}2-1\right) - \frac12 \sqrt q,\label{eq34}\\
b &=& \frac12\left(\frac1{m-1} + \ell + \frac{N}2-1\right) + \frac12 \sqrt q,\label{eq35}\\
q &=& \frac{2\lambda}{m-1} + \left(\ell + \frac{N}2 -1\right)^2 + \frac{N-2}{m-1} + \frac1{(m-1)^2}\label{eq37},\\
c &=& \ell + \frac{N}2\label{eq36}.
\end{eqnarray}
Then the differential equation \eqref{eq10} has two linearly independent solutions $f_1(r)$ and $f_2(r)$. A first solution is of the form
\[
f_1(r)\;=\; r^{\ell} F(a,b;c;r^2),
\]
where $F$ denotes the hypergeometric function defined in \eqref{eq13}. In the case where $c$ is not a positive integer, a second solution $f_2$ is given by
\[
f_2(r)\;=\;  r^{2-\ell-N} F(a+1-c,b+1-c;2-c;r^2).
\]
If $c$ is a positive integer, $c=k\in\N$, then a second solution has the asymptotics
\[
f_2(r)\;\sim\; \left\{\begin{array}{ll}\ln r&\mbox{for }k=1\\ r^{\ell + 2-2k}&\mbox{for }k\ge2\end{array}\right\}\quad\mbox{as }0<r\ll1,
\]
and
\[
f_2'(r)\;\sim\; r^{\ell + 1-2k}\quad\mbox{as }0<r\ll1.
\]
\end{lemma}

As the second solution $f_2$ will be discarded as a potential eigenfunction of the eigenvalue problem \eqref{eq10} in the proof of Proposition \ref{P2} below, it suffices at this point to quote only its asymptotic behavior at zero for positive integers $c$.
\begin{proof}[Proof of Lemma \ref{L2}.]
A first solution to \eqref{eq10} can be obtained by identifying a set of values $\gamma$, $a$, $b$, and $c$ for which \eqref{eq38} turns into the form \eqref{eq11}, and setting $f_1(r)= r^{\gamma}F(a,b;c;r^2)$. The form of a second linearly independent solution depends on the particular values of $a$, $b$, and $c$. In the simplest case, the second solution is obtained by an appropriate change of the dependent and independent variables. In some cases, however, the construction of the second solution relies deeply on the theory of hypergeometric functions, and we have to refer to the relevant literature.

\medskip

We start identifying values for $\gamma$, $a$, $b$, and $c$ such that the hypergeometric function $F(a,b;c;z)$ solves \eqref{eq38}. Comparing \eqref{eq38} and \eqref{eq11}, we immediately see that $\gamma(\gamma + N-2) = \mu_{\ell}$. In view of the definition $\mu_{\ell}= \ell(\ell + N-2)$ this enforces either $\gamma=\ell$ or $\gamma = 2 - N - \ell$. We choose $\gamma=\ell$. Moreover, the constant $c$ is determined by $\gamma(=\ell)$ and $N$ only, namely by $c=\ell+\frac{N}2$, i.e., \eqref{eq36}. Finally, the values for $a$ and $b$ can be derived from the identities $a + b - c + 1 =\frac1{m-1}$ and $ab=\frac{\ell - \lambda}{2(m-1)}$. A short computation yields the two solutions $a_{\pm} =b_{\mp} = \frac12\left(\frac1{m-1} + \ell + \frac{N}2-1\right) \pm \frac12 \sqrt q$, where $q = \frac{2\lambda}{m-1} + (\ell + \frac{N}2 -1)^2 + \frac{N-2}{m-1} + \frac1{(m-1)^2}$, and this quantity is positive since $\lambda$ is positive. We choose $a=a_-=b_+$ and $b=a_+=b_-$, i.e., \eqref{eq34}--\eqref{eq37}. Observe that $c$ is not a non-positive integer since $N\ge 1$ and $\ell\ge0$, so that $F(a,b;c;z)$ is well defined for $a$, $b$, and $c$ as above.

\medskip

Depending on the particular values of $a,b$ and $c$, we will in the following compute a second linearly independent solution of \eqref{eq10}. This solution will again be of the form $f_2(r) = r^{\ell} F(\tilde a,\tilde b;\tilde c;r^2)$, where $\tilde F:= F(\tilde a,\tilde b;\tilde c;r^2)$ is a second linearly independent solution of the hypergeometric equation \eqref{eq11}. Indeed, computing the Wronskian $W(f_1,f_2)$ yields
\begin{eqnarray*}
W(f_1,f_2)&=& f_1' f_2 - f_1f_2'\\
&=& 2r^{2\ell+1}( F'\tilde F - F\tilde F')\\
&=& 2r^{2\ell+1} W(F,\tilde F)\;\not=\; 0
\end{eqnarray*}
since $F$ and $\tilde F$ are linearly independent.

\medskip

 In our derivation of $\tilde F$, we follow \cite[Ch.\ 8]{BealsWong10}. If $c$ is not a positive integer, then a second linearly independent solution to the hypergeometric differential equation \eqref{eq11} is given by $z^{1-c}F(a+1-c,b+1-c;2-c;z)$ (cf.\ \cite[eq.\ (8.2.6)]{BealsWong10}) and thus $f_2(r)=r^{2-\ell-N} F(a+1-c,b+1-c;2-c;r^2)$ is a second linearly independent solution to \eqref{eq10}. Notice that $f_1$ coincides with $f_2$ in the case $c=1$ and the latter function is not even defined for larger integer values of $c$. For integers $c=k\ge1$, a second solution to the hypergeometric differential equation is given by
\begin{eqnarray*}
\lefteqn{G(a,b;k;z) \;=\; \frac{(\log z)F(a,b;k;z)}{ \Gamma(a+1-k)\Gamma(b+1-k)(k-1)!   }}\\
& +& \!\!\sum_{j=0}^{\infty} \frac{(a)_j (b)_j}{(k)_j j!} \frac{\theta(a+j) + \theta(b+j) - \theta(j+1) - \theta(k+j)}{ \Gamma(a+1-k)\Gamma(b+1-k)(k-1)!   } z^j\\
&+&\!\!  \frac{(-1)^k(k-2)!}{\Gamma(a)\Gamma(b)} \sum_{j=0}^{k-2}\frac{(a+1-k)_j (b+1-k)_j}{(2-k)_j j!} z^{j+1-k},
\end{eqnarray*}
where $\theta(s)=\frac{\Gamma'(s)}{\Gamma(s)}$ and the convention that the last term is zero if $k=1$, cf.\ \cite[eq.\ (8.4.4)]{BealsWong10}. We remark that $F(a,b;k;0)=1$. Inspecting the asymptotic behavior at the origin and eventually redefining $G(a,b;k;z)$ by multiplying by $\Gamma(a)$ and/or $\Gamma(b)$ if $a$ and/or $b$ is a nonpositive integer, we see that
\[
G(a,b;k;z)\;\sim\; \left\{\begin{array}{ll}\ln z&\mbox{for }k=1\\ z^{1-k}&\mbox{for }k\ge2\end{array}\right\}\quad\mbox{as }0<z\ll1.
\]
and
\[
\partial_z G(a,b;k;z)\;\sim\; z^{-k}\quad\mbox{as }0<z\ll1.
\]
(See also the discussion on page 275 in \cite{BealsWong10}.) It remains to set $f_2(r)=r^{\ell} G(a,b;k;r^2)$ and notice that $k=1$ precisely if $\ell=0$ and $N=2$. This proves Lemma \ref{L2}.
\end{proof}


\medskip

We are now in the position to present the

\begin{proof}[Proof of Proposition \ref{P2}.] We consider the eigenvalue problem $\Ha_{\ell} f = \lambda f$ for $f\in\D(\Ha_{\ell})$. By Proposition \ref{P1}, we already know that every eigenvalue $\lambda$ must be positive. It is readily checked that the eigenvalue equation for $\Ha_{\ell}$ is equivalent to the differential equation \eqref{eq10}. This equation is solvable for every $\lambda>0$ by Lemma \ref{L2}, and a solution $f$ to \eqref{eq10} is an eigenfunction of $\Ha_{\ell}$ if and only if $f\in\D(\Ha_{\ell})$. Moreover, $f$ can be written as a linear combination of the linearly independent solutions $f_1$ and $f_2$. As solutions of \eqref{eq10} are automatically smooth in $(0,1)$ by interior regularity results, we have to study the asymptotic behavior of $f_1$ and $f_2$ at the singular points $0$ and $1$ in order to decide whether $f_1, f_2\in\D(\Ha_{\ell})$.

\medskip

For further references we quote that
\[
\partial_z F(a,b;c;z)\;=\; \frac{ab}c F(a+1,b+1;c+1;z),
\]
cf.\ \cite[eq.\ (8.2.3)]{BealsWong10}, and therefore
\begin{equation}\label{eq29}
f_1'(r)\;=\; \ell r^{\ell-1} F(a,b;c;r^2) + 2\frac{ab}c r^{\ell+1} F(a+1,b+1;c+1;r^2),
\end{equation}
and, provided $c$ is not a positive integer,
\begin{eqnarray}
f_2'(r)&=& (2-\ell - N)r^{1-\ell-N}F(a-c+1,b-c+1,2-c;r^2)\nonumber\\
&&\mbox{} + 2\frac{(a-c+1)(b-c+1)}{2-c}r^{3-\ell-N}F(a-c+2,b-c+2;3-c;r^2).\label{eq30}
\end{eqnarray}
Moreover, one easily computes that
\begin{equation}\label{eq32}
F(a,b;c;0)\;=\; 1
\end{equation}
if $c$ is not a non-positive integer.

\medskip

{\em Step 1. It holds $f_2\not\in H^1_{\ell}$, and thus $f_2$ is not an eigenfunction of $\Ha_{\ell}$.}\\
In the one-dimensional case, we can easily rule out $f_2$ as an eigenfunction of $\Ha_{\ell}$ in view of the symmetry properties: If $\ell=0$ then $f_2$ is an odd function and if $\ell=1$ then $f_2$ is even. We now turn to the multidimensional case. We first consider the case where $c=k$ is an integer. Then, by Lemma \ref{L2} it is $f_2'(r)\sim r^{\ell+1-2k}$ asymptotically at the origin, and thus, using $2k=2\ell+N$ by \eqref{eq36}
\[
\rho_*(r)(f_2'(r))^2 r^{N-1}\;\sim\; r^{2\ell+1-4k+N}\; =\; r^{1-2k}\quad\mbox{as }0<r\ll1.
\]
The term on the right is not integrable for any $k\ge1$, and thus $\|f_2\|_{H^1_{\ell}}=\infty$. This shows that $f_2\not\in H_{\ell}^1$ in the case where $c=k$ is an integer. If $c$ is not an integer, then $c=\ell +\frac{N}2>1$ because $N\ge 2$. Our argument is based on \eqref{eq30} and \eqref{eq32}. In fact, denoting by $C_{a,b,c}$ a constant that depends only on $a$, $b$, and $c$, we have $f_2'(r)\approx (2-\ell-N)r^{1-\ell-N} +  C_{a,b,c} r^{3-\ell-N}\sim r^{1-\ell-N}$ asymptotically at the origin.  Here we have used that $\ell+N\not=2$ because $c>1$. Therefore
\[
\rho_*(r)(f_2'(r))^2 r^{N-1}\;\sim\; r^{1-2\ell-N}\; =\; r^{1-2c}\quad\mbox{as }0<r\ll1.
\]
Again, the right hand side is not integrable since $c>1$ and thus $f_2\not\in H_{\ell}^1$ in the case where $c$ is not an integer.

\medskip

{\em Step 2. For $-a\not\in \N_0$, it holds $f_1\not\in\D(\Ha_{\ell})$, and thus $f_1$ is not an eigenfunction of $\Ha_{\ell}$.}\\
To rule out solutions $f_1$ as eigenfunctions of $\Ha_{\ell}$ in the case $-a\not\in \N_0$, we investigate the asymptotics at the singularity $r=1$. Since solutions $f_1$ are smooth in $(0,1)$, the asymptotic boundary conditions in the definition of $\D(\Ha_{\ell})$ can be explicitly checked via \eqref{40}. In fact, eigenfunctions must necessarily satisfy $\lim_{r\uparrow1} \rho_*(r) f'(r)=0.$

\medskip

In the following, we show that this condition fails if $-a\not\in\N\cup\{0\}$. In the case that $a+b-c\not\in\Z$, we have the linear transformation formula
\begin{eqnarray*}
F(a,b;c;z) &=& \frac{\Gamma(c)\Gamma(c-a-b)}{\Gamma(c-a)\Gamma(c-b)} F(a,b;a+b-c+1;1-z)\\
&&\mbox{} +(1-z)^{c-a-b}\frac{\Gamma(c)\Gamma(a+b-c)}{\Gamma(a)\Gamma(b)} F(c-a,c-b;c-a-b+1;1-z),
\end{eqnarray*}
cf.\ \cite[15.3.6]{Abramowitz}. In view of \eqref{eq32} and $a+b-c\not\in\Z$, it is
\[
F(a,b;c;z)\;\approx\; \frac{\Gamma(c)\Gamma(c-a-b)}{\Gamma(c-a)\Gamma(c-b)} + (1-z)^{c-a-b} \frac{\Gamma(c)\Gamma(a+b-c)}{\Gamma(a)\Gamma(b)}\quad \mbox{as } 0<1-z\ll1.
\]
Invoking \eqref{eq29}, $b>0$, and $a+b-c+1=\frac1{m-1}$, we deduce that
\[
 \rho_*(r)f_1'(r) \;\sim\; 2\frac{\Gamma(c)\Gamma(a+b-c+1)}{\Gamma(a)\Gamma(b)}  (1-r^2)^{\frac1{m-1}+c-a-b-1}\;\sim\;1\quad \mbox{as } 0<1-r\ll1,
\]
and thus \eqref{40} is violated. Consequently, $f_1$ is not in $\D(\Ha_{\ell})$, and thus not an eigenfunction of $\Ha_{\ell}$.

\medskip

As the above linear transformation has a pole when $a+b-c\in\Z$, we have to use substitute formulas to investigate the limiting behavior at $z=1$. Such transformation formulas can be found in \cite[15.3.10, 15.3.12]{Abramowitz}. Instead of displaying the explicit expressions in the sequel, we will just quote the limiting behavior of $F(a,b;c;z)$ as $z$ converges to $1$. We note that $c<a+b+1$ since $a+b-c+1=\frac1{m-1}>0$. We have
\[
 F(a,b;a+b;z)\;\approx\; -\frac{\Gamma(a+b)}{\Gamma(a)\Gamma(b)} \ln(1-z)\quad \mbox{as } 0<1-z\ll1,
\]
and for $k\in \N$,
\[
 F(a,b;a+b-k;z)\;\approx\; \frac{\Gamma(k)\Gamma(a+b-k)}{\Gamma(a)\Gamma(b)} (1-z)^{-k}\quad \mbox{as } 0<1-z\ll1.
\]
For every $k\in\N\cup\{0\}$, the case $a+b-c =k $ is equivalent to $\frac1{m-1} = k+1$, and thus, a short computation using \eqref{eq29} yields
\[
\rho_*(r)f_1'(r)\;\sim\; (1-r^2)^{\frac1{m-1} - (k+1)}=1\quad \mbox{as } 0<1-r\ll1.
\]
This violates \eqref{40} and thus $f_1$ is not an eigenfunction of $\Ha_{\ell}$.

\medskip

{\em Step 3. For $-a=k\in \N_0$, it holds $f_1\in\D(\Ha_{\ell})$, and thus $f_1$ is an eigenfunction of $\Ha_{\ell}$. The corresponding eigenvalue is $\lambda = \ell + 2k + 2k\left(k+\ell + \frac{N}2 - 1\right)(m-1)$.}\\
We observe that the hypergeometric series in \eqref{eq13} terminates, namely
\[
F(-k,b;c;z) = \sum_{j=0}^k \frac{(a)_j(b)_j}{(c)_j j!} z^j.
\]
It follows that $f_1$ is a polynomial of degree $2k+\ell$ and thus a continuous function on $[0,1]$. In particular, the asymptotic boundary condition \eqref{40} is trivially satisfied. Regarding the integrability of $f_1$, we only have study its asymptotic behavior at $r=0$ in the case where $\mu_{\ell}\not=0$, i.e., $\ell\ge 1$. In this case, we have that $(f_1(r))^2 r^{N-3}\approx r^{2\ell +N-3}$, which is integrable, and thus $f_1\in H^1_{\ell}$. (For the one-dimensional case, we additionally observe that $f_1$ is even if $\ell=0$ and odd if $\ell=1$.) Moreover, since $\Ha_{\ell}f_1=\lambda f_1$ we also know that $\Ha_{\ell}f_1\in H^1_{\ell}$, and thus $f_1\in \D(\Ha_{\ell})$. We conclude that $f_1$ is an eigenfunction of $\Ha_{\ell}$ if $a$ is a nonpositive integer.

\medskip

It remains to express the eigenvalue $\lambda$ in terms the fixed variables. A straight forward computation based on the expression \eqref{eq34} yields that
\[
\lambda = \ell + 2k + 2k\left(k+\ell + \frac{N}2 - 1\right)(m-1).
\]
This concludes the proof of Proposition \ref{P2}.
\end{proof}


\section*{Appendix: Elliptic theory in $H_{\rho_*}^1$}

This appendix provides some results related to elliptic problems in $H_{\rho_*}^1$ that apply in the discussion of the displacement Hessian $\Ha$ in Section \ref{S2.1}. The statements in the Lemmas \ref{L1}--\ref{L5} below are the analogs of very classical facts in standard Sobolev theory: density of smooth functions, a Hardy--Poincar\'e inequality, a Rellich lemma, the Poisson problem. Our Hilbert space $\H$ differs from the well-studied Sobolev space $H^1(B_1)$ only by the weight $\rho_*$, which is, by the regularity of the Barenblatt profile in the ball $B_1$, a very mild variation: $\rho_*$ is finite and degenerates only on the boundary $\partial B_1$. In this regard, it is not surprising that the results in $\H$ can be proved analogously to or based on the known theory.

\medskip

We first give an overview of the main facts. We start with a classical density statement.

\begin{lemma}[Density of smooth functions]\label{L1}
For every $\psi\in H_{\rho_*}^1$, there exists a sequence $\left\{\psi_{\nu}\right\}_{\nu\in\N}$ in $C^{\infty}(\bar B_1)$ such that
\[
\lim_{\nu\uparrow\infty}\int \rho_* |\grad(\psi-\psi_{\nu})|^2\, dx\;=\;0.
\]
\end{lemma}

The next result is an embedding theorem between weighted Lebesgue and Sobolev spaces.

\begin{lemma}[Hardy--Poincar\'e inequality]\label{L4}
Let $p>1$ such that $p+m\ge3$. Then
\begin{equation}\label{eq21}
\inf_{c\in\R}\int \rho_*^{p-m} (\psi-c)^2\, dx\;\lesssim\; \int \rho_* |\grad\psi|^2\, dx
\end{equation}
for any $\psi\in H_{\rho_*}^1$.
\end{lemma}
Notice that the infimum in the statement of the Hardy--Poincar\'e inequality is attained by $c=\int \rho_*^{p-m} \psi\, dx/\int \rho_*^{p-m}\, dx$. The inequality is sharp if $p+m=3$. In most cases, we will apply \eqref{eq21} with $p=2$. The general result is used in the proof of the following compactness result.

\begin{lemma}[Rellich lemma]\label{L7}
The embedding of $\H$ in $ L^2(B_1,\rho^{2-m}_*dx)$ is compact.
\end{lemma}

We finally study the Poisson problem in $\H  $.

\begin{lemma}[Poisson problem]\label{L5}
For every $u\in L^2(B_1,\rho_*^{m-2}dx)$ with $\int u\, dx=0$ there exists a unique (up to an additive constant) $\psi\in H^2_{\loc}(B_1)\cap\H$ such that
\begin{eqnarray*}
-\div (\rho_* \grad\psi) &= &u\quad\mbox{in }B_1\\
\rho_* \grad \psi\cdot\nu&= &0\quad\mbox{on }\partial B_1.
\end{eqnarray*}
\end{lemma}

As before, we understand the boundary conditions in the above lemma in the sense of \eqref{eq52}.

\medskip

We first address the proof of Lemma \ref{L1}, which uses classical approximation techniques, see e.g., \cite[Sec.\ 5.3]{Evans}.
\begin{proof}[Proof of Lemma \ref{L1}.]
We fix $\psi\in H_{\rho_*}^1$ and notice that that $\psi\in H^{1}(K)$ for every compact set $K\subset B_1$ since $\rho_*$ is positive and finite away from $\partial B_1$. 
%
%
In a first step, we show that $\psi$ can be approximated by $C^{\infty}( B_1)$ functions. For this purpose, we consider $A_0= B_{\frac23}$ and annuli $A_{\nu}= B_{\frac{\nu+2}{\nu+3}}\setminus B_{\frac{\nu}{\nu+1}}$ for every $\nu\in\N$, where $B_r$ denotes the ball of radius $r$ around the origin,
to the effect of $B_1\;=\; \bigcup_{\nu=0}^{\infty} A_{\nu}$. Let $\left\{\eta_{\nu}\right\}_{\nu\in\N_0}$ be a partition of unity subordinate the covering $\left\{A_{\nu}\right\}_{\nu\in\N_0}$ and let $\left\{\zeta_{\mu}\right\}_{\mu\in\N_0}$ be a sequence of standard mollifiers. We fix $\eps>0$ arbitrarily. Since $\eta_{\nu}\psi\in H^1(A_{\nu})$, for any $\nu\in\N_0$ there exists a $\mu=\mu(\eps,\nu)\in\N_0$ such that
\[
\left(\int \rho_* |\grad\left(\eta_{\nu}\psi - \left(\eta_{\nu}\psi\right)\ast\zeta_{\mu}\right)|^2\, dx\right)^{1/2}\;\le\; \frac{\eps}{2^{\nu+1}}
\]
and
\[
\supp\left(\left(\eta_{\nu}\psi\right)\ast\zeta_{\mu}\right) \subset B_{\frac{\nu+3}{\nu+4}}\setminus B_{\frac{\nu-1}{\nu}},
\]
with $B_0=\emptyset$. We define $ \tilde\psi = \sum_{\nu=0}^{\infty} \psi_{\nu}$ where $\psi_{\nu} = \left(\eta_{\nu}\psi\right)\ast\zeta_{\mu}$. We obviously have $\tilde\psi\in C^{\infty}( B_1)$. Moreover, since $\psi=\sum_{\nu=0}^{\infty} \eta_{\nu}\psi$, we have for every compact set $K\subset B_1$:
\begin{eqnarray*}
\left(\int_K \rho_* |\grad(\psi - \tilde \psi)|^2\, dx\right)^{1/2} &\le& \sum_{\nu=0}^{\infty} \left(\int_K \rho_* |\grad(\eta_{\nu}\psi - \psi_{\nu})|^2\, dx\right)^{1/2} \\
&\le&\sum_{\nu=0}^{\infty} \frac{\eps}{2^{\nu+1}}\;=\;\eps.
\end{eqnarray*}
Taking the supremum over all $K$, this shows that $C^{\infty}(B_1)\cap \H$ is dense in $H_{\rho_*}^1$.

\medskip

In order to prove density of functions that are smooth up to the boundary, we may, due to the above argumentation, suppose that $\psi\in C^{\infty}(B_1)\cap \H$. For $x\in B_1$ and $\nu\in\N$, we define $\psi_{\nu}(x)=\psi((1-\frac1{\nu})x)$. It follows that $\psi_{\nu} \in C^{\infty}(\bar B_1)$ and
\[
\int \rho_* |\grad (\psi - \psi_{\nu})|^2\, dx\; \stackrel{\nu\uparrow\infty}{\longrightarrow}\; 0
\]
by the dominated convergence theorem.
\end{proof}

We now turn to the proof of Lemma \ref{L4}, which relies on the following one-dimensional inequality:

\begin{lemma}[Double-weight Hardy inequality]\label{L3}
Let $p>1$ such that $p+m\ge3$. Then
\begin{equation}\label{eq18}
\inf_{c\in\R}\int_0^1 (1-r)^{\frac{p-m}{m-1}} (\psi-c)^2 r^{N-1}\, dr\; \lesssim\; \int_0^1 (1-r)^{\frac1{m-1}} (\psi')^2 r^{N-1}\, dr
\end{equation}
for any $\psi\in C^1[0,1]$.
\end{lemma}

Notice that the above inequality is optimal for $p+m=3$ in the case $m\not=2$. For $p=1$ and $m=2$, however, this estimate fails (logarithmically) as the underlying Hardy inequality is critical.
 
\begin{proof}[Proof of Lemma \ref{L4}.] Since $C^{\infty}(\bar B_1)$ is dense in $H_{\rho_*}^1$ by Lemma \eqref{eq1}, it is enough to consider the case where $\psi$ is smooth up to the boundary. In the case $N=1$, \eqref{eq21} is an immediate consequence of \eqref{eq18} and the definition of $\rho_*$ in \eqref{eq51}. We thus concentrate on the case $N\ge 2$. We prove a slightly stronger statement by choosing $c=\int \rho^{p-1}_*|x|^{-1} \psi\, dx/\int \rho^{p-1}_*|x|^{-1}\, dx$, or equivalently by assuming that
\begin{equation}\label{eq25}
\int \rho_*^{p-1}|x|^{-1} \psi\, dx\;=\;0.
\end{equation}
Notice that the integral is well-defined for every $\psi\in C^{\infty}(\bar B_1)$ by the assumption on $p$ and since $N\ge2$. We rewrite the statement in spherical coordinates: Let $x=r\omega$ and $\hat\psi(r,\omega)=\psi(x)$ where $r=|x|$ and $\omega=x/|x|$. Then $|\grad \psi|^2 = (\partial_r \hat\psi)^2 + \frac1{r^2}|\grad_{\S^{N-1}} \hat \psi|^2$, where $\grad_{\S^{N-1}}$ denotes the tangential gradient and thus, by the definition of $\rho_*$, \eqref{eq21} follows from
\begin{eqnarray}
\lefteqn{\int_0^1 \int_{\S^{N-1}} (1-r)^{\frac{p-m}{m-1}} \hat\psi^2 r^{N-1}\, d\omega dr}\nonumber\\
& \lesssim & \int_0^1 \int_{\S^{N-1}} (1-r)^{\frac1{m-1}} \left((\partial_r\hat\psi)^2 + \frac1{r^2} |\grad_{\S^{N-1}}\hat\psi|^2\right)r^{N-1}\, d\omega dr.\label{eq23}
\end{eqnarray}
For any fixed $\omega$, we consider $\psi_{\omega}(r) = \hat\psi(r,\omega)$. Since $\psi_{\omega}\in C^{\infty}[0,1]$, we can apply Lemma \ref{L3} componentwise for any $\omega\in\S^{N-1}$. Observe that the infimum in \eqref{eq18} is attained by a constant $c\sim\int_0^1 (1-r)^{\frac{p-m}{m-1}} \psi_{\omega} r^{N-1}\, dr$, so that \eqref{eq18} and integration over $\S^{N-1}$ yield
\begin{eqnarray*}
\lefteqn{\int_{\S^{N-1}}\int_0^1 (1-r)^{\frac{p-m}{m-1}} \hat\psi^2 r^{N-1}\, drd\omega}\\
&\lesssim&\int_{\S^{N-1}}\left(\int_0^1 (1-r)^{\frac{p-m}{m-1}} \hat\psi r^{N-1}\, dr\right)^2 d\omega+ \int_{\S^{N-1}}\int_0^1 (1-r)^{\frac1{m-1}} (\partial_r\hat\psi)^2 r^{N-1}\, drd\omega.
\end{eqnarray*}
To derive \eqref{eq23}, we need to control the first term on the right. Integrating by parts and applying the Cauchy--Schwarz inequality, we compute
\begin{eqnarray*}
\lefteqn{\left|\int_0^1 (1-r)^{\frac{p-m}{m-1}} \hat\psi r^{N-1}\, dr\right|}\\
&=& \left|\frac{m-1}{p-1}(N-1) \int_0^1 (1-r)^{\frac{p-1}{m-1}}\hat\psi r^{N-2}\, dr+\frac{m-1}{p-1}\int_0^1 (1-r)^{\frac{p-1}{m-1}} \partial_r\hat\psi r^{N-1}\, dr \right|\\
&\lesssim& \left|\int_0^1 (1-r)^{\frac{p-1}{m-1}}\hat\psi r^{N-2}\, dr\right|\\
&&\mbox{} + \left(\int_0^1 (1-r)^{\frac{2p-3}{m-1}} r^{N-1}\, dr\right)^{1/2}\left(\int_0^1 (1-r)^{\frac1{m-1}} (\partial_r\hat\psi)^2 r^{N-1}\, dr\right)^{1/2}.
\end{eqnarray*}
Observe that the second term is finite by the assumption on $p$ and $m$. Now that the above estimate becomes
\begin{eqnarray*}
\lefteqn{\int_{\S^{N-1}}\int_0^1 (1-r)^{\frac{p-m}{m-1}} (\hat\psi)^2 r^{N-1}\, drd\omega}\\
&\lesssim&\int_{\S^{N-1}}\left(\int_0^1 (1-r)^{\frac{p-1}{m-1}} \hat\psi r^{N-2}\, dr\right)^2 d\omega+ \int_{\S^{N-1}}\int_0^1 (1-r)^{\frac1{m-1}} (\partial_r\hat\psi)^2 r^{N-1}\, drd\omega,
\end{eqnarray*}
and the statement in \eqref{eq23} follows from the estimate
\begin{equation}\label{eq22}
\int_{\S^{N-1}}\left(\int_0^1 (1-r)^{\frac{p-1}{m-1}} \hat\psi r^{N-2}\, dr\right)^2 d\omega
\;\lesssim\;
\int_{\S^{N-1}} \int_0^1 (1-r)^{\frac1{m-1}} |\grad_{\S^{N-1}}\hat\psi|^2r^{N-3}\, drd\omega.
\end{equation}
To prove \eqref{eq22}, we start by recalling the Poincar\'e inequality on a sphere (cf.\ \cite[Theorem 2.10]{Hebey99})
\[
\int_{\S^{N-1}} \xi^2\, d\omega \;\lesssim \left(\int_{\S^{N-1}} \xi\, d\omega \right)^2 + \int_{\S^{N-1}} |\grad_{\S^{N-1}}\xi|^2\, d\omega ,
\]
which we apply to $\xi= \int_0^1 (1-r)^{\frac{p-1}{m-1}}\hat\psi r^{N-2}\, dr$ to the effect of
\begin{eqnarray*}
\lefteqn{\int_{\S^{N-1}} \left(\int_0^1 (1-r)^{\frac{p-1}{m-1}}\hat\psi r^{N-2}\, dr\right)^2 d\omega}\\
&\lesssim& \left(\int_{\S^{N-1}} \int_0^1 (1-r)^{\frac{p-1}{m-1}}\hat\psi r^{N-2}\, dr d\omega \right)^2\\
&&\mbox{} + \int_{\S^{N-1}} \left(\int_0^1 (1-r)^{\frac{p-1}{m-1}}|\grad_{\S^{N-1}}\hat\psi| r^{N-2}\, dr\right)^2d\omega.
\end{eqnarray*}
The first term on the right vanishes since
\[
\int_{\S^{N-1}} \int_0^1 (1-r)^{\frac{p-1}{m-1}} \hat\psi r^{N-2}\, dr d\omega\;\sim\;\int \rho_*^{p-1}|x|^{-1} \psi\, dx\;\stackrel{\eqref{eq25}}{=}\;0.
\]
The second term can be estimated using the Cauchy--Schwarz inequality in the inner integral:
\begin{eqnarray*}
\lefteqn{\int_{\S^{N-1}} \left(\int_0^1 (1-r)^{\frac{p-1}{m-1}} |\grad_{\S^{N-1}}\hat\psi| r^{N-2}\, dr\right)^2 d\omega}\\
&\le& \int_0^1 (1-r)^{\frac{2p-3}{m-1}} r^{N-1}\, dr \int_{\S^{N-1}}\int_0^1 (1-r)^{\frac1{m-1}}  |\grad_{\S^{N-1}}\hat\psi|^2 r^{N-3}\, drd\omega.
\end{eqnarray*}
Since the prefactor is bounded by the assumption on $p$ and $m$, this proves \eqref{eq22}.
\end{proof}

\begin{proof}[Proof of Lemma \ref{L3}.] The statement \eqref{eq18} follows as a combination of the two Hardy inequalities
\begin{equation}\label{eq19}
\int_0^{1/2} (\psi-\psi(1/2))^2 r^{N-1}\, dr\;\lesssim\; \int_0^{1/2} (\psi')^2 r^{N+1}\, dr
\end{equation}
and
\begin{equation}\label{eq20}
\int_{1/2}^1 (1-r)^{\frac{p-m}{m-1}} (\psi-\psi(1/2))^2\, dr\;\lesssim\; \int_{1/2}^1 (1-r)^{\frac1{m-1}} (\psi')^2\, dr.
\end{equation}
Indeed,
\begin{eqnarray*}
\lefteqn{\int_0^1 (1-r)^{\frac{p-m}{m-1}} (\psi-\psi(1/2))^2 r^{N-1}\, dr}\\
 & \lesssim& \int_0^{1/2} (\psi-\psi(1/2))^2 r^{N-1}\, dr + \int_{1/2}^1 (1-r)^{\frac{p-m}{m-1}} (\psi-\psi(1/2))^2 \, dr\\
&\stackrel{\eqref{eq19}\&\eqref{eq20}}{\lesssim}& \int_0^{1/2} (\psi')^2 r^{N+1}\, dr + \int_{1/2}^1 (1-r)^{\frac1{m-1}} (\psi')^2\, dr\\
&\lesssim& \int_0^1 (1-r)^{\frac1{m-1}} (\psi')^2 r^{N-1}\, dr,
\end{eqnarray*}
and the term on the left can be bounded from below by taking the infimum over all $c=\psi(1/2)\in\R$.

\medskip

For proving \eqref{eq19} and \eqref{eq20}, it is enough to consider the case where $\psi(1/2)=0$. The proof of the Hardy inequality \eqref{eq19} is standard:
\begin{eqnarray*}
\int_0^{1/2} \psi^2 r^{N-1}\, dr &=& \left. \frac1N \psi^2 r^N \right|_{r=0}^{r=1/2} - \frac2N\int_0^{1/2} \psi\psi' r^N\, dr\\
&\le&\frac2N \left(\int_0^{1/2} \psi^2 r^{N-1}\, dr \right)^{1/2}\left(\int_0^{1/2} (\psi')^2 r^{N+1}\, dr \right)^{1/2}.
\end{eqnarray*}
We apply the Young inequality $2ab\le a^2+b^2$ to infer \eqref{eq19}.

\medskip

Likewise, for \eqref{eq20} we have
\begin{eqnarray*}
\lefteqn{\int_{1/2}^1 (1-r)^{\frac{p-m}{m-1}} \psi^2\, dr}\\
 &=& -\left.\frac{m-1}{p-1}(1-r)^{\frac{p-1}{m-1}} \psi^2 \right|_{r=1/2}^{r=1} + 2\frac{m-1}{p-1}\int_{1/2}^1 (1-r)^{\frac{p-1}{m-1}} \psi\psi'\, dr\\
&\le&2 \frac{m-1}{p-1}\left(\int_{1/2}^1 (1-r)^{\frac{2p-3}{m-1}}\psi^2\, dr\right)^{1/2}\left(\int_{1/2}^1 (1-r)^{\frac1{m-1}} (\psi')^2\, dr\right)^{1/2}.
\end{eqnarray*}
In order to deduce \eqref{eq20}, it remains to check that $(1-r)^{\frac{2p-3}{m-1}}\le (1-r)^{\frac{p-m}{m-1}}$ holds if $p\ge3-m$, and to apply the Young inequality.
\end{proof}

We now address Lemma \ref{L7}.

\begin{proof}[Proof of Lemma \ref{L7}.]
We derive the statement from the well-known analogous statement for regular Sobolev spaces, cf.\ \cite[p.\ 272]{Evans}. Let $\{\psi_n\}_{n\in\N}$ denote a bounded sequence in $\H$. We may without lost of generality assume that
$
\int \rho_*^{p-m}\psi_n\, dx \;=\;0,
$
where $p<2$ denotes a constant that satisfies the hypothesis of Lemma \ref{L4} (this is possible since $m>1$), and thus the sequence is uniformly bounded in $L^2(B_1, \rho_*^{p-m}dx)$ by \eqref{eq21}. As $\H$ is a separable Hilbert space, we may extract a subsequence that converges weakly to a limit $\psi $ both in $\H$ and in $L^2(B_1,\rho_*^{p-m}dx)$. Moreover, as $\rho_*$ is bounded below by a positive constant in every compact subset of $B_1$, we have that the sequence $\{ \psi_n\}_{n\in \N}$ is bounded in the (unweighted) Sobolev space $H^1(B_{1-\frac1k})$ for every $k\in\N$. By the Rellich compactness lemma, we may successively extract further subsequences $\{\psi_{\tau_k(n)}\}_{n\in\N}$ that converge to $\psi$ strongly in $L^2(B_{1-\frac1k})$. By the boundedness of $\rho_*$, we may choose a diagonal sequence $\{\psi_{\tau(k)}\}_{k\in\N} $ such that
\begin{equation}\label{eq59}
\int_{B_{1-\frac1k}} \rho_*^{2-m} (\psi-\psi_{\tau(k)})^2\, dx \;\le\; \frac{1}{k},
\end{equation}
for all $k\in\N$. For $\eps>0$ arbitrary but fixed, we write
\[
\int \rho_*^{2-m} (\psi-\psi_{\tau(k)})^2\, dx \;=\; \int_{B_{1-\eps}} \rho_*^{2-m} (\psi-\psi_{\tau(k)})^2\, dx + \int_{B_1\setminus B_{1-\eps}} \rho_*^{2-m} (\psi-\psi_{\tau(k)})^2\, dx.
\]
The first integral converges to zero if $k$ goes to infinity thanks to \eqref{eq59}. For the second one, we have that
\[
 \int_{B_1\setminus B_{1-\eps}} \rho_*^{2-m} (\psi-\psi_{\tau(k)})^2\, dx\;\lesssim\; \eps^{\frac{2-p}{m-1}}\int_{B_1} \rho_*^{p-m} (\psi-\psi_{\tau(k)})^2\, dx,
\]
and integrals on the right are bounded uniformly in $k$. Hence, letting $k$ converge to infinity, we have that 
\[
\lim_{k\uparrow\infty}\int \rho_*^{2-m} (\psi-\psi_{\tau(k)})^2\, dx \;\le\; C\eps^{\frac{2-p}{m-1}},
\]
for some $C>0$. Since $\eps$ was arbitrary and $p<2$, this proves the statement of Lemma \ref{L7}.
\end{proof}

The proof of Lemma \ref{L5} is very classical and we just sketch it.

\begin{proof}[Proof of Lemma \ref{L5}.] 
Observe that the elliptic problem is the Euler--Lagrange equation of the strictly convex functional
\[
\F(\psi)\;=\; \frac12 \int \rho_* |\grad\psi|^2\, dx - \int u\psi\, dx
\]
for $\psi\in H_{\rho_*}^1$, and thus, it suffices to prove existence of a minimizer of $\F$ in $H_{\rho_*}^1$. Indeed, by the strict convexity of $\F$, minimizers $\psi$ are unique up to additive constants and satisfy the weak Euler--Lagrange equation
\begin{equation}\label{eq53}
\int \rho_* \grad\psi\cdot\grad\xi\, dx\;=\; \int u\xi\, dx\quad\mbox{for all }\xi\in\H.
\end{equation}
In particular, choosing $\xi\in C_c^{\infty}(B_1)$, we see that $\psi$ is a distributional solution of $-\div\left(\rho_*\grad\psi\right)=u$, which also holds in the strong sense by interior regularity estimates. Moreover, with this information, \eqref{eq53} becomes
\[
\int\rho_* \grad\psi\cdot\grad\xi\, dx\;=\;-\int \div\left(\rho_*\grad\psi\right)\xi\, dx\quad\mbox{for all }\xi\in\H,
\]
i.e., $\rho_*\grad\psi\cdot\nu=0$ on $\partial B_1$. 

\medskip

To prove existence of minimizers of $\F$ in $\H$, we just mention the basic steps, following the direct method. We consider a minimizing sequence for $\F$ in $H_{\rho_*}^1$. With the help of Lemma \ref{L4}, it is easy to check that this sequence is bounded in $H_{\rho_*}^1$. This fact is based on the estimate
\begin{eqnarray*}
\int u\psi\, dx&=&\int u(\psi-c)\, dx\\
&\le&\left(\int \rho_*^{m-2}u^2\, dx\right)^{1/2}\left( \int \rho_*^{2-m}(\psi-c)^2\, dx\right)^{1/2}\\
&\lesssim& \left(\int \rho_*^{m-2}u^2\, dx\right)^{1/2}\left( \int \rho_*|\grad\psi|^2\, dx\right)^{1/2},
\end{eqnarray*}
where $c$ denotes the optimal constant of Lemma \ref{L4}, that could be introduced in the first identity thanks to the fact that $u$ has average zero. As a consequence, we can find a weakly converging subsequence. By lower semicontinuity of $\F$ with respect to weak convergence, we immediately deduce that the minimum of $\F$ in $H_{\rho_*}^1$ is attained.
\end{proof}

\section*{Acknowledgements}The author thanks Robert McCann for bringing this topic to his attention and for enlightening discussions during the process of this study. He gratefully acknowledges helpful suggestions from the anonymous referee.

\bibliographystyle{elsarticle-harv}
\bibliography{pme_lit}

\end{document}